\documentclass[11pt]{amsart}
\usepackage{amsmath}
\usepackage{amssymb}
\usepackage{amsthm}
\usepackage{xcolor}
\numberwithin{equation}{section}

\newcommand{\ud}{\,\mathrm{d}}
\newcommand{\udiv}{\, \mathrm{div}}
\newcommand{\eps}{\varepsilon}
\newcommand{\R}{{\mathbb{R}}}
\newcommand{\ind}{{\mathbb{I}}}
\newcommand{\T}{\Pi}

\theoremstyle{plain}
\newtheorem{Prop} {Proposition} 
\newtheorem{lemma} {Lemma}
\newtheorem{theorem}{Theorem}

\newtheorem{definition}{Definition}
\theoremstyle{remark} \newtheorem*{remark}{Remark}                                                    %\theoremstyle{remark}\newtheorem*{example}{Remark}

 \numberwithin{theorem}{section}
 \numberwithin{Prop}{section}
 \numberwithin{corollary}{section}
\numberwithin{lemma}{section}
\numberwithin{definition}{section}

\title[Mean Field limits]{Mean Field Limit and Quantitative Estimates with singular attractive kernels }

\author[D. Bresch]{Didier Bresch}
\address{Laboratoire de Math\'ematiques, CNRS UMR 5127,
Universit\'e  Savoie Mont-Blanc, 73376 Le Bour\-get-du-Lac, France; 
e-mail: Didier.Bresch@univ-smb.fr}

\author[P.-E. Jabin]{Pierre-Emmanuel Jabin}
\address{Pennsylvania State University, Department of Mathematics and
Huck Institutes, State College, PA 16802, USA; 
e-mail: pejabin@psu.edu}

\author[Z. Wang]{Zhenfu Wang} 
\address{Beijing International Center for Mathematical Research, Peking University, Beijing, China, 100871; 
e-mail:  zwang@bicmr.pku.edu.cn}

\date{\today}

\begin{document}

\begin{abstract} This paper proves the mean field limit and quantitative estimates for many-particle systems with singular attractive interactions between particles.   As an important example,  a full rigorous derivation (with quantitative estimates) of the  Patlak-Keller-Segel model  in optimal subcritical regimes is obtained for the first time.  To give an answer to this longstanding problem, we take advantage of a new modulated free energy and we prove some precise large deviation estimates encoding the competition between diffusion and attraction. Combined with the range of repulsive kernels, already treated  in the s\'eminaire Laurent Schwartz proceeding  [https://slsedp.centre-mersenne.org/journals/SLSEDP/ ], we provide the full proof of results announced by the authors in [{\it C.~R. Acad. Sciences}, \rm  Section Maths, (2019)].
\end{abstract} 

\maketitle

\tableofcontents  

\section{Introduction}
\label{sec2}
\noindent
The present paper coupled with the proceeding [S\'eminaire Laurent Schwartz, EDP et Applications, ann\'ee 2019-2020, Expos\'e no II] published on website: https://slsedp.centre-mersenne.org/journals/SLSEDP/   corresponds to the extended version ({\it i.e.} with detailed proofs) of the announced results in the note {\it C.R. Acad. Sciences} \cite{BrJaWa}.

Using a new weighted related entropy, we are able to derive for the first time the mean field limit for many-particle systems with singular attractive interactions of gradient-flow type. In particular we can positively answer the long standing open question of the mean field limit to the Patlak-Keller-Segel system. 

More precisely, we consider the mean field limit for stochastic many-particle systems of the type
\begin{equation}\label{sys}
d X_i = \frac{1}{N} \sum_{j\not = i} K(X_i-X_j)  dt+  \sqrt{2\sigma}  d B_i, \quad i=1, 2, \cdots, N, 
\end{equation}
where the $B_i$ are independent Brownian Motions or Wiener processes. For simplicity, we limit ourselves to the periodic domain $\T^d$.

We specifically focus on gradient flows with 
  interaction kernels given by
 \begin{equation}\label{kernel}
 K= -\nabla V
 \end{equation}
 with general singular and {\em attractive} interaction potentials $V$.   

A guiding example in this article (and the corresponding note \cite{BrJaWa}) is the {\em attractive} Poisson potential in dimension $2$
 \begin{equation}
   V=\lambda\,\log |x|+V_e(x),\label{logpotential}
 \end{equation}
 with $\lambda>0$ and where $V_e$ is a smooth correction so that $V$ is periodic. Logarithmic potentials still play a critical role if the dimension $d>2$ and for this reason we will still consider potentials like \eqref{logpotential} in any dimension, even if there is no connection with the Poisson equation anymore.

Our main goal is to provide precise quantitative estimates 
for the convergence of \eqref{sys} towards the limit McKean-Vlasov PDE 
\begin{equation}
  \begin{split}
& \partial_t \bar \rho + \udiv_x \left( \bar  \rho\,  K \star_x \bar \rho \right) = \sigma \Delta_x \bar \rho,   \\ 
&\hbox{ with } \quad K = - \nabla V,\quad \bar\rho(t=0,x)=\bar\rho^0\in {\mathcal P}(\T^d).
\end{split} \label{MFlimit}
\end{equation}

In the case where $V$ is given by \eqref{logpotential} and $d=2$, \eqref{MFlimit} is the famous Patlak-Keller-Segel model, which is one of the first models of chemotaxis for micro-organisms. The potential $-V \star \bar \rho $ can then be seen as the concentration of some chemical  species (one has typically $V\leq 0$ here): From \eqref{logpotential}, one has that $\Delta V-V=2\pi\,\lambda\, \delta_0$ so that the chemical species are produced by the population. Moreover \eqref{MFlimit} implies that the population follows the direction of higher chemical concentrations (more negative values of $V$).

\smallskip

It should be noted that the system~\eqref{MFlimit} offers only a rough modeling of the biological processes involved in chemotaxis. For realistic applications, it is hence critical to be able to handle a {\em wide range} of potentials $V$ that may still share a singularity comparable to the one in \eqref{logpotential}. In that sense, the Patlak-Keller-Segel model is a good example of a typical setting where a singular attractive dynamics (all micro-organisms try to concentrate on a point) is competing with the spreading effect due to diffusion.

There exists a further advantage of considering \eqref{logpotential}: Since \eqref{MFlimit} has a simple structure, it is possible to fully characterize its behavior. Eq. \eqref{MFlimit} may indeed blow-up and form a Dirac mass in finite time and, one  may exactly characterize that such a blow-up occurs simply by comparing $\lambda$ and $\sigma$
\begin{itemize}
\item If $\lambda\leq 2\,d\,\sigma$ with $d$ the dimension, then we always have a global solution to \eqref{MFlimit};
  \item If $\lambda>2\,d\,\sigma$ then all solutions blow-up in final time (though it may be possible to extend the existence of some notion of solution past some blow-up as in \cite{BedMas}).
  \end{itemize}
We refer for instance to \cite{BlDolPer,DoSe,DoPe} and the references therein. We note here that in our case $\bar \rho^0$ is normalized to be a probability density with total mass $1$. The PDE literature typically instead normalizes $V$ to be the Green kernel of the Poisson equation, so that the result above exactly corresponds to the classical $8\,\pi\,\sigma$ critical mass.

A key consequence of our main result is that $V\sim \log |x|$ is always critical for the mean field limit. As we will see in our main result, Theorem~\ref{Main}, we are able to prove the limit for essentially all $V\geq \gamma\,\log |x|$ for some $\gamma<2\,d\,\sigma$ (with some reasonable assumptions on $\nabla V$, see \eqref{Vlp}-\eqref{nablaV} below). This justifies the central role played by the Patlak-Keller-Segel case but it is remarkable that the exact same condition is found for the mean field limit as for the blow-up of the PDE system.

\smallskip

There are several ways to quantitatively compare \eqref{sys} with the limit $\bar\rho$ given by \eqref{MFlimit}, which one can very roughly separate into trajectorial and statistical approaches.  We follow here \cite{JaWa1,JaWa2} by using the joint law $\rho_N(t,x_1,\ldots,x_N)$ of the process $(X_1,\ldots,X_N)$ which solves the Liouville or forward Kolmogorov equation 
\begin{equation}
  \begin{split}
 &   \partial_t \rho_N
  + \sum_{i=1}^N {\rm div}_{x_i} \Bigl(\rho_N \frac{1}{N} \sum_{j \ne i }^N K(x_i-x_j)\Bigr) 
  = \sigma \sum_{i=1}^N \Delta_{x_i} \rho_N,
  \\
 & \rho_N\vert_{t=0} = \rho_N^0.
  \end{split} \label{liouvilleN}
\end{equation}
Eq. \eqref{liouvilleN} contains all the relevant statistical information about the position of the particles at any time. But it may fail to include some information on the trajectories: For example, it is not in general possible to identify time correlations in a given particle trajectory only from \eqref{liouvilleN}. 

We also emphasize that, in addition to quantitative convergence estimates, Eq. \eqref{liouvilleN} also offers a straightforward manner to understand solutions to system \eqref{sys}. We are actually not able to give a precise meaning to trajectorial solutions to the original SDE system~\eqref{sys}. But instead we will be working here with so-called {\em entropy solutions} to \eqref{liouvilleN}, which are more straightforward to define (see the appendix for an example). Of course any strong solution to \eqref{sys} (in the probabilistic sense) would  also yield an entropy solution to \eqref{liouvilleN}. 
  
 \smallskip
 
The joint law $\rho_N$ is compared to the chaotic/tensorized law
$\bar \rho_N := \bar \rho^{\otimes N}=\Pi_{i=1}^N \bar\rho(t,x_i)$, built from the limit $\bar \rho$.
Of course $\bar\rho_N$ cannot be an exact solution to \eqref{liouvilleN} so the method will have to account for the difference.
%but  instead solves
%\begin{equation}
%  \begin{split}
%&\partial_t \bar \rho_N + \sum_{i=1}^N \udiv_{x_i}  \Bigl(\bar \rho_N  \, K\star_x \bar \rho(x_i) \,  \Bigr) 
%  = \sigma\, \sum_{i=1}^N \Delta_{x_i} \bar \rho_N \\ 
%&    \bar\rho_N\vert_{t=0} = (\bar{\rho}^0)^{\otimes N}
%  \end{split}  \label{tensorRhoN}
%\end{equation}
As probability densities, both $\rho_N$ and $\bar\rho_N$ are initially normalized by
\begin{equation} \label{Norm}
 \int_{\T^{dN}} \rho_N\vert_{t=0} = 1 =  \int_{\T^{d}} \bar  \rho_N\vert_{t=0},
\end{equation}
which is formally preserved by either \eqref{liouvilleN} or  \eqref{MFlimit}.%/\eqref{tensorRhoN}.
  The method leads in particular to direct estimates between $\bar\rho^{\otimes k}$ and any  observable or marginal of the system at a fixed rank $k$,
\[
\rho_{N,k}(t,x_1,\ldots,x_k)=\int_{\T^{(N-k)\,d}} \rho_N(t,x_1,\ldots,x_N)\,dx_{k+1}\ldots d x_{N}.
\]
We postpone a full presentation of our main result till the main section as this requires a more technical discussion of the method. Still, a good example of corollary from our more complete Theorem~\ref{Main} is a rigorous derivation of the Patlak-Keller-Segel system in the subcritical regime as given by
\begin{theorem} Assume that $\rho_N\in L^\infty(0,T;\;L^1(\T^{Nd}))$ is an entropy solution to Eq. \eqref{liouvilleN} normalized by \eqref{Norm} in the sense of Definition \ref{defentro}, with initial condition $\rho_N(t=0)=\bar\rho^{\otimes N}(t=0)$, and for the potential $V$ given by \eqref{logpotential}. Assume that  $\bar\rho \in L^\infty(0,T;W^{2,\infty}(\Pi^d))$ solves Eq. \eqref{MFlimit}  with $\inf \bar\rho >0$. Assume finally that $\lambda<2d\,\sigma$. Then there exists a constant $C>0$ and an exponent $\theta>0$ independent of $N$
  s.t. for any fixed $k$ 
\[
\|\rho_{N,k}-\bar\rho^{\otimes k}\|_{L^\infty(0,T;\;L^1(\T^{kd}))}\leq C\,k^{1/2}\,N^{-\theta}.
\]\label{PKSestimate}
\end{theorem}
Theorem \ref{PKSestimate} follows directly from Theorem \ref{Main} will be stated below and the classical Csisz\'ar-Kullback-Pinsker inequality.  The exponent $\theta$ could be made fully explicit and actually depends only on $2d\,\sigma-\lambda$.   We highlight that we obtain, in dimension 2, the optimal constant $4\sigma$ which corresponds to the critical mass $8\pi\sigma$ for which we have blow-up in finite time for the  Patlak-Keller-Segel system.  

\smallskip

Because of the singular behavior of the potential, a full rigorous derivation  of the Patlak-Keller-Segel model from the stochastic equation \eqref{sys} or the Liouville eq.~\eqref{liouvilleN} had remained elusive, in spite of recent progress in \cite{CaPe,FoJo} or \cite{GQ,HaSc}. In particular, the results in \cite{FoJo} prove that any accumulation point as $N\to \infty$ of the random empirical measure associated to the system \eqref{sys} is a weak solution in some sense to \eqref{MFlimit} provided that one is in the so-called very subcritical regime with $\lambda<\sigma$. While this provides the mean field limit, at least in some weak sense, it does not imply propagation of chaos. We also emphasize that \cite{FoJo} is also able to obtain well-posedness for the original SDE system~\eqref{sys} in the same regime $\lambda<\sigma$. Of course the case of regularized Patlak-Keller-Segel interactions is much better understood with for example \cite{OliRicTom}.
%  Of course potentials like \eqref{logpotential} are  only one example of singular interactions between particles for which the mean field limit remains poorly understood, especially in the stochastic cases.  

\medskip

The paper is organized as follows: In Section \ref{NRE}, we introduce the new relative entropy with weights related to the Gibbs equilibrium $G_N$ and the corresponding distribution $G_{\bar\rho_N}$ given in \eqref{Gibbs}: modulated free energy. We also state our main quantitative Theorem  \ref{Main} based on this modulated energy and we describe the main steps of the proof for reader's convenience.  We finally present the explicit expression for the time evolution of such modulated free energy first in Proposition \ref{modulatedsmooth} for smooth solutions associated to smooth kernels and then Proposition \ref{modulatedweak} for entropy solutions associated to singular kernels.
In Section \ref{LDE}, we present various large deviation type estimates which play crucial roles in the proof of Theorem \ref{Main} to control the non-negativity of the modulated energy. 
For the reader's convenience, we conclude with an appendix which recalls some previous large deviation estimates in \cite{JaWa2} and proves the existence of entropy solutions for the Liouville equation~\eqref{liouvilleN} for the Patlak-Keller-Segel interaction kernel in two dimension.
%%%%%%%%%%%%%%%%%%%%%%%%%%%%%%%%%%%%%%%%%%%%%%%%%%%%%%%%%%%%%%
\section{New modulated free energy and main quantitative result \label{NRE}}
%%%%%%%%%%%%%%%%%%%%%%%%%%%%%%%%%%%%%%%%%%%%%%%%%%%%%%%%%%%%%
\subsection{Weighted relative entropy and the modulated free energy}
%%%%%%%%%%%%%%%%%%%%%%%%%%%%%%%%%%%%%%%%%%%%%%%%%%%%%%%%%%%%
The method will revolve around the control of a  rescaled entropy combining the relative entropy by Jabin-Wang  \cite{JaWa1} and the modulated energy by Serfaty \cite{Se} and Duerinckx \cite{Du}. This corresponds to a modulated free energy for the problem and reads 
\begin{equation}
\begin{split}
& E_{N} \Big(\frac{\rho_N}{G_N}\,|\;\frac{\bar \rho_N}{G_{\bar \rho_N}} \Big)    \\
  &\ =\frac{1}{N} %\Bigl[
  \int_{\Pi^{dN}} \rho_N (t,X^N) \log \Bigl(\frac{\rho_N(t,X^N) }{G_N(t,X^N)}\frac{G_{\bar  \rho_N}(t,X^N)}{\bar \rho_N(t,X^N)}\Bigr) dX^N,   \\ 
%& \hskip2cm -  \int_{\pi^{dN}} \rho_N (t,X^N) \, dX^N
 %  +\int_{\Pi^{dN}}  \frac{\bar \rho_N(t,X^N)} {G_{\bar \rho_N}(t,X^N)}  G_N(t,X^N\bigr)
  %      dX^N\Bigr]
\end{split}\label{modulateddef}
\end{equation}
where
\begin{equation}
\begin{split}\label{Gibbs}
& G_N(t, X^N) = \exp \bigg(- \frac{1}{2N\sigma } \sum_{i \ne  j} V (x_i-x_j)\bigg), \\ 
& G_{\bar\rho}(t, x)=\exp \bigg(-\frac{1}{\sigma}\,V\star \bar \rho(x) +\frac{1}{2\,\sigma}\,\int_{\Pi^{d} } V\star \bar \rho \, \bar\rho\bigg),  \\
& G_{\bar \rho_N} (t,X^N) = \exp \bigg(-  \frac{1}{\sigma} \sum_{i=1}^NV\star \bar \rho(x_i) 
     +  \frac{N}{2\sigma} \int_{\Pi^{d} } V\star \bar \rho \, \bar\rho \bigg),
\end{split}
\end{equation}
and throughout this article $X^N=(x_1, x_2, \cdots, x_N)$.
 This free energy may be understood as a relative entropy with two weights
(related to the Gibbs equilibrium) in the spirit of \cite{BrJa}. Note that
\[
E_{N} \Big(\frac{\rho_N}{G_N}\,|\;\frac{\bar \rho_N}{G_{\bar \rho_N}} \Big) 
   = {\mathcal H}_N(\rho_N\vert \bar \rho_N) 
    + {\mathcal K}_{N} (G_N\vert G_{\bar \rho_N}) 
    %+ {\mathcal L}_{\bar \rho_N/G_{\bar \rho_N}} (G_N \vert G_{\bar \rho_N})
    ,
    \]
where
\begin{equation}\label{HN}
{\mathcal H}_N(\rho_N\vert \bar \rho_N)
   = \frac{1}{N} \int_{\Pi^{dN}} \rho_N(t,X^N) \log\Bigl(\frac{\rho_N(t,X^N)}{\bar \rho_N(t,X^N)}\Bigr)
      \, d X^N 
\end{equation}
is exactly the relative entropy introduced in \cite{JaWa1,JaWa2} and 
\begin{equation}\label{KN}
 {\mathcal K}_{N} (G_N\vert G_{\bar \rho_N}) 
    = - \frac{1}{N}  \int_{\Pi^{dN}} \rho_N(t,X^N) \log \bigl(\frac{G_N(t,X^N)}{G_{\bar \rho_N}(t,X^N)} \bigr)
      \, d X^N
 \end{equation}
with $G_N$ and $G_{\bar\rho_N}$ given by Expressions \eqref{Gibbs} is the expectation of the modulated energy on which the method developed  in \cite{Du,Se,Se1} is based.  Indeed, it is easy to check that 
      \[
      \mathcal{K}_N(G_N \vert G_{\bar \rho_N} ) = \frac{1}{ 2 \sigma} \int_{\Pi^{d N}} \ud \rho_N \int_{\Pi^{2d} \cap \{x \ne y \}} V(x - y) (d \mu_N -  d \bar \rho)^{\otimes 2 }(x, y). 
      \]
Note that  $E_N$ is not a priori a positive quantity.  Since given any  two measures $f$ and $g$, not necessarily probability measures, by convexity one has  $\int (f \log f/g + g - f ) d x \geq 0$, which gives a lower bound for $E_N$, 
\[
E_{N} \Big(\frac{\rho_N}{G_N}\,|\;\frac{\bar \rho_N}{G_{\bar \rho_N}} \Big)  \geq \frac{1}{N} \int_{\Pi^{d N}} \bigg( \rho_N(t, X^N)  - G_N(t, X^N) \frac{\bar \rho_N(t, X^N) }{G_{\bar \rho_N} (t, X^N) }\bigg) d X^N, 
\]
or alternatively using Lemma 1 in \cite{JaWa2} 
\[
\begin{split}
& E_N = \mathcal{H}_N  -  \frac{1}{N} \int \rho_N \log \frac{G_N}{G_{\bar \rho_N} }  \geq  - \frac 1 N \log \int \bar \rho_N \frac{G_N}{G_{\bar \rho_N}}.  
\end{split} 
\]
As we mentioned already, the method in Theorem~\ref{Main} combines the methods developed in \cite{Du,Se,Se1} and \cite{JaWa1,JaWa2} (see also the summary in \cite{Sa}). The modulated energy in \cite{Se,Se1}  proved effective for the mean field limit for Coulomb or Riesz gases, and was able to take advantage of the specific structure of the interaction to improve on previous results; though for less general interaction than \cite{Hauray} for example. On the other hand, the relative entropy in \cite{JaWa1,JaWa2}  could not effectively handled gradient flows but performed well on interaction kernels that have bounded divergence with or without diffusion. In particular \cite{JaWa1,JaWa2} obtained quantitative mean field estimates from the 2d viscous model to the incompressible Navier-Stokes vs. previously only qualitative results in \cite{FHM,Osada86,Osada}.

While relative entropy at the level of the Liouville equation  have not been widely used for mean field limits, the relative entropy method initiated in  \cite{Yau}  is maybe the closest.  A different relative entropy approach at the level of the joint law of the full trajectories of the system was also developed in \cite{BenArousZeitouni}.

As mentioned in \cite{BrJaWa}, combining the relative entropy with a modulated energy has already been very successfully used for various singular limits in kinetic theory. A first example concerns the so-called quasineutral limit for plasmas for which we refer for instance to \cite{Kw} and \cite{PuSa}. Another example is the seminal derivation of the incompressible viscous Electro-magneto-hydrodynamics from the Vlasov-Maxwell-Boltzmann system in \cite{ArSa}; one issue in that monograph is in particular to prove the asymptotic positivity of the combined free energy, which is a problem that we are facing as well as explained below.  Specific tools are needed for the present (and different) context of the mean field limit for many-particle systems. Of course the role of the free energy for gradient flow systems has long been recognized, with \cite{CaJuMaToUn,Ot} being classical examples.  

%The present paper concerns the case with repulsive kernels
%with precise large deviation inequalities to be proved encoding the competition between diffusion and attraction. Coupled with the complementary s\'eminaire Laurent Schwartz paper \cite{BrJaWa1} which concerns the repulsive kernels, this provides the proof for general kernels (attractive-repulsive) announced by the authors in  the {\it C.R. Acad. Sciences} \cite{BrJaWa}: see remark just after 
%the main Theorem \ref{Main}.

%%%%%%%%%%%%%%%%%%%%%%%%%%%%%%%%%%%%%%%%%%%%%%%%%%%
\subsection{The main quantitative theorem}
%%%%%%%%%%%%%%%%%%%%%%%%%%%%%%%%%%%%%%%%%%%%%%%%%%%
Because our method relies on propagating nonlinear quantities such as the relative free energy,   some assumptions on the notion of solutions are required, namely
\begin{definition} \label{defentro} {\bf (Entropy solution)} Let $T>0$ be fixed. A density $\rho_N\in L^\infty(0,T; L^1(\Pi^{dN})$ with $\rho_N\ge 0$ and  $\int_{\Pi^{dN}} \rho_N dX^N= 1$, is an entropy solution to Eq. \eqref{liouvilleN} on the time interval $[0,T]$  if it solves \eqref{liouvilleN} in the sense of distributions, and for {\it a.e.} $t\le T$
\begin{equation}
\begin{split}
&\int_{\Pi^{dN}} \rho_N(t,X^N) \log \Big(\frac{\rho_N(t,X^N)}{G_N} \Big) \, dX^N \\
& + \sigma  \sum_{i=1}^N\int_0^t \int_{\Pi^{dN}} 
    \rho_N(s,X^N) \Bigl|\nabla_{x_i} \log \Big(\frac{\rho_N(s,X^N)}{G_N} \Big) \Bigr|^2  dX^N ds \\
&
     \le \int_{\Pi^{dN}} \rho^0_N  \log \bigl(\frac{\rho^0_N}{G_N}\bigr) \, dX^N \label{entropycontrol} 
\end{split}
\end{equation}
where for convenience we use in the article the notation $X^N = (x_1, \cdots, x_N)$. 
\end{definition}
Because of the singularity in the interaction, a weak solution to \eqref{liouvilleN} may not be an entropy solution in the sense given above. We note however as well that entropy solutions need not be unique and Theorem~\ref{Main} below applies to all entropy solutions if there exists more than one. While it is not the main purpose of this article, we include in the appendix a proof of the existence of entropy solutions for the Patlak-Keller-Segel setting.

Let us now specify the exact assumptions that are required on the potential $V$  
\begin{align}
  &V\in L^p(\Pi^d)\cap C^2(\Pi^d\setminus\{0\})\quad  \mbox{for some}\ p>1  \label{Vlp},\\
  &V(x)\geq \lambda \,\log |x|+C\quad  \mbox{for some}\ 0\leq \lambda<2\,d\,\sigma, \label{Vattractive}\\
  &|\nabla V(x)|\leq \frac{C}{|x|}.\label{nablaV}
\end{align}
with $C>0$ constant. Then the following theorem holds
\begin{theorem}\label{Main}
  Assume $K=-\nabla V$ with $V$ a singular potential that satisfying \eqref{Vlp}--\eqref{nablaV}. Consider $\bar \rho \in L^\infty(0,T;W^{2,\infty}(\Pi^d))$ solves Eq. \eqref{MFlimit}  with $\inf\bar\rho>0$. Assume finally that $\lambda<2d\,\sigma$. There exists  constants $C>0$ and $\theta>0$ s.t. for $\bar \rho_N=\Pi_{i=1}^N \bar \rho(t, x_i)$, and for the joint law $\rho_N$ on $\Pi^{dN}$ of any entropy solution to the SDE system \eqref{sys}, 
\[\begin{split}
& {\mathcal H}_N(t)+ |\mathcal{K}_N(t)| \leq e^{C_{\bar\rho}\,\|K\| \,t}\,\Big({\mathcal H}_N(t=0)+ |\mathcal{K}_N(t=0)|+\frac{C}{N^\theta}\Big),\\
\end{split}\]
where $H_N$ and $\mathcal{K}_N$ are defined by \eqref{HN} and \eqref{KN}, $C_{\bar \rho}$ and $\|K\|$ are  constants depends on $\bar \rho$ and the assumptions \eqref{Vlp}--\eqref{nablaV} on $V$ respectively.  
Hence if $H_N^0+|\mathcal{K}_N^0|\leq C\,N^{-\theta}$, then for any fixed marginal $\rho_{N,k}$ 
\[
\|\rho_{N,k}-\Pi_{i=1}^k \bar\rho(t, x_i)\|_{L^1(\Pi^{k\,d})}\leq C_{T,\bar\rho,k}\,N^{-\theta}.
\]
\end{theorem}

\bigskip
\begin{remark}
Note that the same result may be obtained for attractive-repulsive kernels combining the present paper for attractive kernels to the method detailed in \cite{BrJaWa1} for repulsive kernels.  This corresponds
to the announced results in the note {\it C.R. Acad. Sciences} \cite{BrJaWa}. More precisely we can get Theorem \ref{Main} for an even kernel $V$ which may be decomposed as follows $V=V_a+V_r + V_s$ with $V_a$ an attractive part satisfying \eqref{Vlp}--\eqref{nablaV}, $V_s\in W^{2,\infty}(\Pi^d)$ a smooth part and $V_r$ a repulsive part satisfying
\begin{equation} \label{hyp01}
V_r(-x)=V_r(x) \quad \hbox{ and } \quad  V_r\in L^p(\Pi^d) \hbox{ for some } p>1
\end{equation}
with the following Fourier sign
\begin{equation}\label{hyp02}
\hat V_r(\xi) \ge 0 \hbox{ for all } \xi \in \R^d.
\end{equation}
One imposes  the following pointwise controls for all $x\in \T^d$: There exists constants $k$ and $C>0$
such that
\begin{equation}\label{hyp03}
 |\nabla V_r(\xi)| \le  \frac{C}{|x|^k}, \qquad 
 |\nabla^2 V_r(x)|\le \frac{C}{|x|^{k}},\quad |\nabla V_r(x)| \le C\frac{V_r(x)}{|x|},
  \end{equation}
together with 
\[
\lim_{|x|\to 0} V_r(x) = +\infty, \qquad V_r(x) \le C V_r(y) \hbox{ for all } |y|\le 2|x|,
\]
  and
  \[
  |\nabla_\xi \hat V_r(\xi) \le \frac{C}{1+|\xi|} \Bigl( \hat V_r(\xi) + \frac{1}{1+|\xi|^{d-\alpha}}\Bigr)
  \hbox{ with } 0<\alpha <d \hbox{ for all } \xi \in \R^d. 
  \]
  \end{remark}
%%%%%%%%%%%%%%%%%%%%%%%%%%%%%%%%%%%%%%%%%%%%%%%%%%%
\subsection{Proof of Theorem~\ref{Main}: the main steps}
%%%%%%%%%%%%%%%%%%%%%%%%%%%%%%%%%%%%%%%%%%%%%%%%%%%%%%%%%%
We describe here the various steps to prove Theorem~\ref{Main}.  It follows the general strategy detailed in \cite{BrJaWa1} which was dedicated to repulsive kernels but with some key differences due to the attractive singularity.

$\bullet$ {\em Step~1: The  modulated free energy inequality.} 
The first step is of course to look at the time evolution of our modulated free energy. It is possible to show that it satisfies the following inequality
\begin{equation}
\begin{split}
& E_{N}\left(\frac{\rho_N}{ G_{N}}\,\vert\;\frac{\bar\rho_N}{G_{\bar\rho_N}}\right)(t) \le E_{N}\left(\frac{\rho_N}{ G_{N}}\,\vert\;\frac{\bar\rho_N}{G_{\bar\rho_N}}\right)(0)
\\
&\hskip.2cm  -\frac{1}{2} \int_0^t\int_{\Pi^{dN}} \int_{\Pi^{2\,d}\cap \{x\neq y\}}  \nabla V(x-y) \cdot \\
 &\hskip2cm \Big(\nabla\log \frac{\bar\rho}{G_{\bar \rho}}(x)
 - \nabla\log \frac{\bar\rho}{G_{\bar \rho}}(y)\Big) (d\mu_N - d\bar\rho)^{\otimes 2} d\rho_N.
\end{split} \label{propfreeenergy}
\end{equation}
The inequality~\eqref{propfreeenergy} exactly corresponds to the free energy inequality \eqref{ModulatedFreeEnergy} that we prove later in Proposition~\ref{modulatedweak} in subsection~\ref{energynonsmoothV}. This inequality involves two difficulties: The formal calculations themselves are rather intricate and use in a critical manner the properties of the Gibbs equilibrium. The second issue is of course to justify those formal calculations for entropy solutions.
   For this reason we first explain the formal calculations in subsection~\ref{energysmoothV} in Proposition~\ref{modulatedsmooth} for smooth solutions, before presenting in Proposition~\ref{modulatedweak} the argument to extend the inequality to entropy solutions.

\medskip

$\bullet$ {\em Step~2: Control on the right-hand side.} Contrary to the repulsive case, the control on the right-hand side can immediately be obtained from \cite{JaWa2}. Precisely since we assumed in \eqref{nablaV} that 
$|\nabla V(x)|\leq {C}/{|x|}$, we have that 
\[
- \nabla V(x-y) \cdot (\phi(x) - \phi(y)) \in L^\infty, 
\] 
with $\phi(x) = \nabla \log \frac{\bar \rho}{G_{\bar \rho}}(x)$, thus we may directly apply Lemma~1 as in \cite{JaWa2}. We recall that this lemma reads
\begin{lemma}
  For any $\rho_N,\;\bar\rho_N$ in $\mathcal{P}(\Pi^{dN})$, any function $\psi\in L^\infty(\Pi^{d\,N})$ and any $\alpha>0$
  \[
\int_{\Pi^{dN}} \psi(X^N)\,d\rho_N\leq \frac{1}{\alpha\,N}\,\int_{\Pi^{dN}} d\rho_N\,\log \frac{\rho_N}{\bar\rho_N}+\frac{1}{\alpha\,N}\,\log\int_{\Pi^{dN}} e^{\alpha\,N\,\psi(X^N)}\,d\bar\rho_N. 
  \]\label{duality}
  \end{lemma}
For completeness, we give a simple proof of  Lemma \ref{duality} in the appendix.
Applying this lemma, we directly find that
\[
\begin{split}
&- \int_{\Pi^{dN}}\int_{\{x\not=y\}} \nabla V(x-y) \cdot (\phi(x)-\phi(y)) 
(d\mu_N - d\bar\rho)^{\otimes 2}\, d\rho_N\leq C\,\mathcal{H}_N(\rho|\;\bar\rho_N)\\
& +\frac{C}{N} \int_{\Pi^{dN}} d\bar\rho_N \exp\bigg(-\frac{N}{C}\int_{\{x\not=y\}} \nabla V(x-y) \cdot (\phi(x)-\phi(y)) 
(d\mu_N - d\bar\rho)^{\otimes 2}\bigg),
\end{split}
\]
where we denote $\phi(x)=\nabla \log ({\bar\rho}/{G_{\bar\rho}} )(x)$.
We now apply a simplified version of  Theorem~4 in \cite{JaWa2} which reads
\begin{theorem} {\rm (Theorem 4 in \cite{JaWa2})}.  Consider $\bar \rho \in \mathcal{P}(\Pi^d)$ and $f \in L^\infty( \Pi^{2d})$. Then there exists a constant $\alpha =\alpha(f)>0$ small enough such that   
\[
\sup_{N \geq 2} \, \int_{\T^{dN}}  \bar \rho^{\otimes N} \exp \bigg(  \alpha N  \int_{\Pi^{2d}} f(x,y)\,(d\mu_N-d\bar\rho)^{\otimes 2}   \bigg) d X^N  \leq  C < \infty.
\]\label{largedeviationJaWamod}
\end{theorem}
Theorem~\ref{largedeviationJaWamod} is a straightforward reformulation of Theorem~4 in \cite{JaWa2} as we recall in the appendix (see Theorem \ref{largedeviationJaWa} there and its following comments).
This proves that
\[
\int_{\Pi^{dN}} d\bar\rho_N \exp\bigg(-\frac{N}{C}\int_{\{x\not=y\}} \nabla V(x-y) \cdot (\psi(x)-\psi(y)) 
(d\mu_N - d\bar\rho)^{\otimes 2}\bigg)\leq C_{\bar\rho}
\]
for some constant $C_{\bar\rho}$ depending on the $W^{2,\infty}$ norm of $\log \bar\rho$ and it subsequently implies that
\begin{equation}\label{Ineg} 
\begin{split} 
E_{N}\Big(\frac{\rho_N}{ G_{N}}\,\vert\;\frac{\bar\rho_N}{G_{\bar\rho_N}}\Big)(t) 
 \le E_{N}\Big(\frac{\rho_N}{ G_{N}}\,\vert\;\frac{\bar\rho_N}{G_{\bar\rho_N}}\Big)(0)  
  +C\,\int_0^t \mathcal{H}_N(\rho_N\,|\;\bar\rho_N) + \frac{C}{N}.
\end{split} 
  \end{equation}

\medskip

$\bullet$ {\em Step 3: A lower bound on ${\mathcal K}_N$ in $E_N$.}
 From \eqref{Ineg}, it remains to find a lower bound control on $E_N = {\mathcal H}_N + {\mathcal K}_N$ with
\[
  {\mathcal K}_N = \frac{1}{2\sigma} \int_{\Pi^{dN}}
    \int_{ \Pi^{2d } \cap \{x\not = y\}} V(x-y) (d\mu_N - d\bar\rho)^{\otimes 2}(x, y)  d \rho_N 
\]
a quantity which is non-necessarily positive or even asymptotically positive for attractive potentials $V$ since $V\leq 0$ and even $V(x)\to-\infty$ as $x\to 0$.  A natural idea would be to try to control ${\mathcal K}_N$ from below by ${\mathcal H}_N$. Unfortunately direct inequalities between ${\mathcal K}_N$ and ${\mathcal H}_N$ do not appear to be true. However it is possible to  compare ${\mathcal K}_N$
to ${\mathcal H}_N$ by splitting the study in two parts: short-range and long-range interactions. Therefore we introduce
\[
V(x)= \lambda V_0(x)+W(x)=V(x)\,\chi(|x|/\eta)+V(x)\,(1-\chi(|x|/\eta)),
\]
where $\chi$ is a smooth function with $\chi(x)=1$ if $x<1/2$ and $\mbox{supp}\,\chi\in [0,1]$ with a parameter $\eta$ which will be chosen later.

\smallskip

\noindent I) {\it The short-range interactions.} This case focuses on the main difficulty which is the singularity of $V$ near $0$. For this reason, we consider general truncated quantity (short-range interactions) of the type
\begin{equation}
\begin{split}
  &F (\mu)= -  \int_{\Pi^{2d}\cap\{x\neq y\}} V_{0}(x-y)\,(\mu(dx)-\bar\rho(x)\,dx)\,(\mu(dy)-\bar\rho(y)\,dy), \\
  & \hbox{ with } V_{0} \in L^p(\Pi^d),\quad |\nabla V_{0}(x)|\leq \frac{C}{|x|^k} \hbox{ with } k>1/2\\
  &  \hbox{ satisfying the inequality: } V_{0} (x-y)\geq  \log |x-y|\,\chi(|x-y|/\eta),
\end{split}\label{Flogexample}
\end{equation}
for $p>1$ and where we emphasize that, from $\chi$,  $\mbox{supp}\, V_{0}\in B(0,\eta)$.
 Of course we will use \eqref{Flogexample} for $\lambda V_0(x) = V(x)\chi(|x|/\eta)$ and we note that from assumptions \eqref{Vlp}--\eqref{nablaV}, such a $V_0$ indeed satisfies the assumptions in \eqref{Flogexample}. However \eqref{Flogexample} applies to $V_0$ with 
$|\nabla V_{0}(x)|\leq {C}/{|x|^k}$ instead of the more restrictive \eqref{nablaV} which is not required for the proposition below. Our goal is to prove the following
\begin{Prop}\label{prop3.1}
  There exists $\eta$ (depending only on $\|V_0\|_{L^p}$) s.t. for any $\gamma<d$, we have for $F$ defined by \eqref{Flogexample} that for some $\theta>0$
  \[\begin{split}
  &\gamma\,\int_{\Pi^{dN}} F(\mu_N)\,\rho_N\,dX^N\leq  \mathcal{H}_N(\rho_N\,\vert\;\bar\rho_N)\\
  &\quad+\frac{C}{N^\theta}\,(\log N+\|\log\bar\rho\|_{W^{1,\infty}}+\eta^{-1}),
\end{split}
\]\label{Frho_Nlog}
for all $N>\bar N$ with $\bar N$ depending only on the dimension and $d-\gamma$.
\end{Prop}
Prop.~\ref{prop3.1} is the main technical result of this article as it extends classical large deviation estimates to singular attractive potentials. Because of this, its proof is also rather intricate and is performed in Section~\ref{LDE}. We also insist that Prop.~\ref{prop3.1} only holds for some $\eta$ small enough which is the reason for the decomposition of the potential $V$ into short and long ranges.

\smallskip

\noindent {II) \it The long-range interactions.} 
The second part consists in controlling the long-range interaction with $W(x) =V(x)\,(1-\chi(|x|/\eta))$.  Because $W$ is actually smooth, this can be done by  rather straightforward contributions.
 Define $G_N^{W}$ and $G_{\bar\rho_N}^{W}$ for $W$ 
in the similar manner as for $V$,
\begin{equation}\label{GG}
\begin{split}
& G_N^W(t,X^N) = \exp \bigg(- \frac{1}{2N\sigma } \sum_{i,j=1}^N W (x_i-x_j)\bigg),\\ & 
G_{\bar\rho}^W(x)=\exp \bigg(-\frac{1}{\sigma}\,W\star \bar \rho(x) +\frac{1}{2\,\sigma}\,\int_{\Pi^{d} } W\star \bar \rho \, \bar\rho\bigg),  \\
& G_{\bar \rho_N}^W (t,X^N) = \exp \bigg(-  \frac{1}{\sigma} \sum_{i=1}^N W\star \bar \rho(x_i) 
     +  \frac{N}{2\sigma} \int_{\Pi^{d} } W\star \bar \rho \, \bar\rho\bigg).
\end{split}
\end{equation} 
We calculate separately the 
evolution in time of the contribution of $W$ in the modulated free energy through
\begin{lemma}\label{regular} 
  For $W \in W^{2,\infty}_{\rm per}$ even and with $\Delta W(0)=0$, one has that
  \[\begin{split}
& - \frac{d}{dt} {\mathcal K}_N(G_N^W\vert G_{\bar\rho_N}^W)
 = \frac{d}{dt}\frac{1}{N}\int_{\Pi^{dN}} \rho_N \log \frac{G^W_{\bar \rho_N}}{G^W_N}\,dX^N\\
  &\quad=\int_{\Pi^{dN}}  d \rho_N\,\int_{\Pi^{2d}} \Delta W(x-y)\,(d\mu_N-d\bar\rho)^{\otimes 2}\\
&\quad  -\frac{1}{\sigma} \int_{\Pi^{dN }}  d \rho_N\,\int_{\Pi^{2d}} \int_{\Pi^d} \nabla W(z-x)\cdot\nabla V(z-y)\,\ud \mu_N(z) \,(d\mu_N-d\bar\rho)^{\otimes 2} (x,y)\\ 
  &\quad-\frac{1}{2\,\sigma}\,\int_{\Pi^{dN}} d \rho_N\int_{\Pi^{2d}} \nabla W(x-y)\,(\nabla V\star \bar \rho(x)-\nabla V\star\bar\rho(y))\,(d\mu_N-d\bar\rho)^{\otimes 2}(x,y). 
  \end{split}
  \]
\end{lemma}
Of course the right-hand side in this lemma involves more derivatives of $W$ than what we observed in \eqref{propfreeenergy} because there is no particular structure in ${\mathcal K}_N(G_N^W\vert G_{\bar\rho_N}^W)$. However since $W$ is smooth, this actually does not matter.
 In particular we observe that the 2nd term in the right-hand side can be rewritten as 
 \[
 \begin{split}
 &  -\frac{1}{\sigma} \int_{\Pi^{dN}} d\rho_N\,\int_{\Pi^{2d}} \int_{\Pi^d} \nabla W(z-x)\cdot\nabla V(z-y)\, \bar  \rho(z) \ud z  \,(d\mu_N-d\bar\rho)^{\otimes^2} \\
   & - \frac{1}{2 \sigma} \int_{\Pi^{dN}}  d \rho_N \int_{\Pi^{2d}} \nabla V(x-y) (\nabla W \star (\mu_N - \bar \rho)(x) - \nabla W \star (\mu_N - \bar \rho)(y) )\\
   &\hskip8cm(d \mu_N -  d \bar \rho)^{\otimes 2}.  
 \end{split}
 \]
 We may then directly use Lemma \ref{duality} and Theorem~\ref{largedeviationJaWamod} as in Step~$2$ to bound
\begin{equation}
 -\frac{d}{dt} {\mathcal K}_N(G_N^{W}\vert G_{\bar\rho_N}^{W})
 \le C {\mathcal H}_N(\rho_N\vert \bar\rho_N) + \frac{C}{N}.\label{boundLRI}
\end{equation}

\medskip

$\bullet$ {\em Step 4: Coercivity of $\tilde E_N$} We write
\[
\tilde E_N=E_N-{\mathcal K}_N(G_N^W\vert G_{\bar\rho_N}^W)=\mathcal{H}_N+\mathcal{K}_N(G_N\vert G_{\bar\rho_N})-{\mathcal K}_N(G_N^W\vert G_{\bar\rho_N}^W),
\]
where we remove the long range interaction from $\mathcal{K}_N$. From our choices and definitions, we have that
\[
\mathcal{K}_N(G_N\vert G_{\bar\rho_N})-{\mathcal K}_N(G_N^W\vert G_{\bar\rho_N}^W)
=- \frac{\lambda}{2\sigma} \,\int_{\Pi^{dN}} F(\mu_N)\,\rho_N\,dX^N,
\]
where $F$ is defined through \eqref{Flogexample} with $\lambda V_0 (x) = V(x) \chi(|x|/\eta)$. Applying Prop.~\ref{prop3.1} with $\gamma = \lambda/2\sigma$ we hence get that
\begin{equation}\label{lower}
\tilde E_N\geq \frac{1}{C}\, {\mathcal H}_N(\rho_N\vert \bar\rho_N)-\frac{C}{N^\theta},
\end{equation}
for some $C>1$ and  $\theta>0$ assuming $\lambda < 2\sigma d$. See also for instance Eq. (26) and (27) in \cite{BrJaWa}.

$\bullet$ {\em Step 5: Conclusion of the proof.}
We are now ready to explain how to conclude the proof of the theorem. 
 Combining \eqref{Ineg} with \eqref{boundLRI} to obtain that
\begin{equation}
\tilde E_N\leq \tilde E_N(t=0)+C\,\int_0^t {\mathcal H}_N(\rho_N\vert \bar\rho_N) + \frac{C}{N}.\label{timetildeEN}
\end{equation}
Inserting this into \eqref{lower}, we deduce by Gronwall's Lemma that
  \begin{equation} \label{HNN}
\frac{1}{C}\, {\mathcal H}_N(\rho_N\vert \bar\rho_N)(t) \leq e^{C\,t}\,\left(\tilde E_N(t=0)+\frac{C}{N^\theta}\right).
  \end{equation} 
   Finally since $W$ is smooth, from Eq. \eqref{boundLRI} for instance, we trivially have that
 \begin{equation}\label{KNW}
|{\mathcal K}_N(G_N^W\vert G_{\bar\rho_N}^W)|(t) \leq
    |{\mathcal K}_N(G_N^W\vert G_{\bar\rho_N}^W)(0)+  C\,\int_0^t {\mathcal H}_N(\rho_N\vert \bar\rho_N) + \frac{C}{N}.
 \end{equation}
 Combining \eqref{lower} with \eqref{timetildeEN} and \eqref{HNN},  we can control $|\mathcal{K}_N(G_N\vert G_{\bar\rho_N})-{\mathcal K}_N(G_N^W\vert G_{\bar\rho_N}^W)|$  which finally allows to derive all estimates in the main theorem.

\bigskip

To fully complete the proof, it only remains to prove our modulated free energy inequality \eqref{propfreeenergy}, Prop.~\ref{prop3.1} and Lemma~\ref{regular}. The inequality \eqref{propfreeenergy} is proved in the next two subsections in two steps: firstly Prop.~\ref{modulatedsmooth} for the formal calculations and finally Prop.~\ref{modulatedweak} for the rigorous derivation in subsection~\ref{energynonsmoothV}. Lemma~\ref{regular} follows mostly the same calculations and is given just after in subsection~\ref{secregular}. The proof of Prop.~\ref{prop3.1} is the main object of section~\ref{LDE}.
%%%%%%%%%%%%%%%%%%%%%%%%%%%%%%%%%%%%%%%%%%%%%%%%%%%%%%%%%%
\subsection{Modulated free energy control for smooth potential~$V$\label{energysmoothV}}
%%%%%%%%%%%%%%%%%%%%%%%%%%%%%%%%%%%%%%%%%%%%%%%%%%%%%%%%%%
Our first step  is to prove the following explicit expression for the time evolution of the modulated  free energy $E_N$ where all interactions are smooth.
\begin{Prop}\label{modulatedsmooth}
  Assume that $V$ is a ${\mathcal C}^2$ even function  and that $\rho_N$ is a classical solution to \eqref{liouvilleN} and $\bar \rho$ solves \eqref{MFlimit}. Then the modulated free energy $E_N$  defined by \eqref{modulateddef} satisfies
\[
\begin{split}
& E_{N}\Big(\frac{\rho_N}{G_{N}}\,\vert\;\frac{\bar\rho_N}{G_{\rho_N}}\Big)(t) \\
& \le 
E_{N}\Big(\frac{\rho_N}{G_{N}}\,\vert\;\frac{\bar\rho_N}{G_{\rho_N}}\Big)(0) -\frac{\sigma}{N}\,\int_0^t \int_{\Pi^{d\,N}} \rho_N\,\left|\nabla\log \frac{\rho_N}{\bar\rho_N}-\nabla\log \frac{G_N}{G_{\bar\rho_N}}\right|^2\\
&\ -\frac{1}{2} \int_0^t  \int_{\Pi^{dN}} \int_{\Pi^{2\,d} \cap \{  x \ne y \}}  \nabla V(x-y) \cdot \left(\nabla\log \frac{\bar\rho}{G_{\bar \rho}}(x)
- \nabla\log \frac{\bar\rho}{G_{\bar \rho}}(y)\right)\\
&\hspace{250pt} (d\mu_N - d\bar\rho)^{\otimes 2} d\rho_N,
\end{split} 
\]
where $\mu_N=\frac{1}{N}\,\sum_{i=1}^N \delta(x-x_i)$ is the empirical measure associated to the point configuration $X^N=(x_1, \cdots, x_N)$. 
\end{Prop}

\begin{remark}
	Note that 
	\[
	\nabla \log \frac{ \bar \rho }{G_{\bar \rho}} (x) = \nabla \log \bar \rho(x) + \frac{1}{\sigma} \nabla V \star \bar \rho(x). 
	\]
	Taking the right derivatives of this term will exactly cancel in the evolution of our modulated free energy the divergence term $\udiv K = - \Delta V$, that is otherwise present in the calculations in the time evolution of $\mathcal{H}_N(\rho_N \vert \bar \rho_N)$. 
\end{remark} 

\begin{proof}
First of all we write $\rho_N$ as a solution to the variant diffusion equation
\[
  \partial_t \frac{\rho_N}{G_N} -\frac{\sigma}{G_N}\,{\rm div} \left(G_N \nabla \frac{\rho_N}{G_N}\right) = 0.
  \]
  We also try to put $\bar\rho_N$ under this form. Of course we have that
  \begin{equation}
\partial_t \bar\rho={\sigma}\,{\rm div}_{x} \left(G_{\bar\rho}\,\nabla \frac{\bar\rho}{G_{\bar\rho}}\right)={\sigma}\,{\rm div}_{x} \left(\bar\rho\,\nabla \log\frac{\bar\rho}{G_{\bar\rho}}\right),\label{selfadjointbarrho}
  \end{equation}
  which we can tensorize trivially into
  \[
\partial_t {\bar\rho_N} -\sigma\,{\rm div} \left(G_{\bar\rho_N} \nabla \frac{\bar\rho_N}{G_{\bar\rho_N}}\right)=0.
  \]
  So we just write
  \[
 \partial_t \frac{\bar\rho_N}{G_{\bar\rho_N}} -\frac{\sigma}{G_N}\,{\rm div} \left(G_N \nabla \frac{\bar\rho_N}{G_{\bar\rho_N}}\right)=R_N, 
   \]
   where
   \[
   \begin{split}
     R_N=&-\frac{\sigma}{G_N}\,\sum_i {\rm div}_{x_i} \left(G_N\,\nabla_{x_i} \frac{\bar\rho_N}{G_{\bar\rho_N}}\right)+\frac{\sigma}{G_{\bar\rho_N}}\,\sum_i {\rm div}_{x_i} \left(G_{\bar\rho_N}\,\nabla_{x_i} \frac{\bar\rho_N}{G_{\bar\rho_N}}\right)\\
     &+\bar\rho_N \partial_t \frac{1}{G_{\bar\rho_N}}.
   \end{split}
   \]
   We now recall the classical entropy-entropy dissipation inequality for self-adjoint diffusion equations: Consider two solutions $u_i$, $i=1,\;2$ to
   \[
\partial_t u_i-\frac{1}{M(x)}\,\mbox{div}_x\left(M(x)\,\nabla_x u_i \right)=0.
\]
Then one has that by differentiating and integrating by parts
\[
\begin{split}
&\frac{d}{dt} \int u_1\,\log \frac{u_1}{u_2}\,M(x)\,dx=\int \left(\left(1+\log \frac{u_1}{u_2}\right)\,\partial_t u_1-\frac{u_1}{u_2}\,\partial_t u_2\right)\,M(x)\,dx\\
&\quad =-\int \left(\nabla_x u_1\cdot \nabla_x \log \frac{u_1}{u_2}-\nabla_x \frac{u_1}{u_2}\cdot\nabla_x u_2\right)\,M(x)\,dx.
\end{split}
\]
By re-arranging the terms, we obtain the usual
\[
\frac{d}{dt} \int u_1\,\log \frac{u_1}{u_2}\,M(x)\,dx=-\int u_1\,M(x)\,\left|\nabla_x \log \frac{u_1}{u_2}\right|^2\,dx.
\]
In our case, this gives immediately that
\begin{equation}
\begin{split}
\frac{d}{dt} E_N
=& -\frac{\sigma}{N}\,\int_0^t \int_{\Pi^{d\,N}} \rho_N\,\left|\nabla\log \frac{\rho_N}{\bar\rho_N}-\nabla\log \frac{G_N}{G_{\bar\rho_N}}\right|^2 \\
& \hskip3cm -\frac{1}{N}\int \frac{\rho_N}{\bar\rho_N}\,G_{\bar\rho_N}\,R_N.
\end{split}\label{dtEN1}
\end{equation}
So the whole point is to handle correctly the terms with $R_N$. Let us start with
just expanding the divergence terms in $R_N$ and getting the trivial cancellation
\begin{equation}
\begin{split}
     R_N= & \sigma\,\sum_i (\nabla_{x_i} \log G_{\bar\rho_N}-\nabla_{x_i} \log G_N)\cdot\nabla_{x_i} \frac{\bar\rho_N}{G_{\bar\rho_N}} \\
  & +\bar\rho_N \partial_t \frac{1}{G_{\bar\rho_N}}.
   \end{split} \label{RN}
\end{equation} 
Hence the remainder above just reads
\[
\begin{split}
  r_N=&\frac{1}{N}\int \frac{\rho_N}{\bar\rho_N}\,G_{\bar\rho_N}\,R_N=\frac{\sigma}{N}\int \rho_N\,\sum_i \nabla_{x_i}\log \frac{\bar \rho_N}{G_{\bar\rho_N}}\cdot\nabla_{x_i}\log \frac{G_{\bar\rho_N}}{G_N}\\
  &-\frac{1}{N}\int \rho_N \partial_t \log G_{\bar\rho_N}.  
\end{split}
\]
This is of course directly
\[
\begin{split}
  r_N=&\frac{1}{N}\int \rho_N\,\int_{\Pi^{2d}} \nabla V(x-y)\cdot\nabla_x \log \frac{\bar \rho_N}{G_{\bar\rho_N}}(x)\,\mu_N(dx)\,(\mu_N-\bar\rho)(dy) \\
  &-\frac{1}{N}\int \rho_N \int_{\Pi^d} \mu_N(dx) \partial_t \log G_{\bar\rho}.  
\end{split}
\]
Use now \eqref{selfadjointbarrho} to get that
\[
\begin{split}
  &\partial_t \log G_{\bar\rho}=-\frac{1}{\sigma}\,V\star\partial_t\bar\rho+\frac{1}{\sigma}\int V\star\bar\rho \partial_t\bar\rho=-\nabla V\star(\bar\rho\,\phi)-\int \nabla V\star\bar\rho\,\bar\rho \phi,
  \end{split}
\]
where we denote $\phi(x)=\nabla_x \log  \bigl({\bar \rho_N}/{G_{\bar\rho_N}}(x)\bigr)$.

If we insert this into $r_N$, we get that
\[
\begin{split}
  r_N=&\frac{1}{N}\int \rho_N\,\int_{\Pi^{2d}} \nabla V(x-y)\cdot \Big(\phi(x)\,\mu_N(dx)\,(\mu_N-\bar\rho)(dy)\\ &\quad +\phi(y)\,\mu_N( dx)\,\bar\rho(dy) +\phi(x)\,\bar\rho(dx)\,\bar\rho(dy)\Big). \\
  \end{split}
\]
It just remains to symmetrize in $x$ and $y$ to obtain 
\begin{equation}
\begin{split}
 & r_N  = \frac{1}{N} \int_0^t\int \frac{\rho_N}{\bar\rho_N} G_{\bar\rho_N} R_N  \\
  & = \frac{1}{2N}\int \rho_N\,\int_{\Pi^{2d}} \nabla V(x-y)  \\
  &\quad  \cdot (\phi(x)-\phi(y))(\mu_N(dx)\,\mu_N(dy)-2\mu_N(dy)\,\bar\rho(dx)+\bar\rho(dx)\,\bar\rho(dy)),  
\end{split} \label{rN}
\end{equation}
which concludes.
\end{proof}

%%%%%%%%%%%%%%%%%%%%%%%%%%%%%%%
\subsection{Modulated free energy control for non-smooth $V$} \label{energynonsmoothV}
%%%%%%%%%%%%%%%%%%%%%%%%%%%%%%%%%%%%%%       
We prove here an equivalent of Prop. \ref{modulatedsmooth} for realistic, singular potentials $V$, which finally implies \eqref{propfreeenergy}.
\begin{Prop}\label{modulatedweak}
  Assume that $V$ is an even function,   and that $\rho_N$ is an entropy solution to \eqref{liouvilleN} in the sense of definition \ref{defentro} below and $\bar \rho$ smooth  solves \eqref{MFlimit}. Then the modulated free energy defined by \eqref{modulateddef} satisfies that
\begin{equation}
\begin{split}
& E_{N}\left(\frac{\rho_N}{ G_{N}}\,\vert\;\frac{\bar\rho_N}{G_{\bar\rho_N}}\right)(t)
  \\
  &\ \le E_{N}\left(\frac{\rho_N}{ G_{N}}\,\vert\;\frac{\bar\rho_N}{G_{\bar\rho_N}}\right)(0)
 -\frac{\sigma}{N}\,\int_0^t\int_{\Pi^{d\,N}} \rho_N\,\left|\nabla\log \frac{\rho_N}{\bar\rho_N}-\nabla\log \frac{G_N}{G_{\bar\rho_N}}\right|^2\\
&\ -\frac{1}{2} \int_0^t\int_{\Pi^{dN}} \int_{\Pi^{2\,d}\cap \{x\neq y\}}  \nabla V(x-y) \cdot \left(\nabla\log \frac{\bar\rho}{G_{\bar \rho}}(x)
 - \nabla\log \frac{\bar\rho}{G_{\bar \rho}}(y)\right)\\
 &\hspace{250pt}(d\mu_N - d\bar\rho)^{\otimes 2} d\rho_N,
\end{split} \label{ModulatedFreeEnergy}
\end{equation}
where $\mu_N=\frac{1}{N}\,\sum_{i=1}^N \delta(x-x_i)$ is  the empirical measure.
\end{Prop} 
\begin{proof}  
  We first note that the entropy estimate provides some useful {\em a priori} estimates. In particular from the entropy dissipation \eqref{entropycontrol}, we have that 
\begin{equation}
\begin{split}
&   \sum_{i=1}^N\int_0^T \int_{\Pi^{dN}} 
    \rho_N(s,X^N) \Bigl|\nabla_{x_i} \log \bigl(\frac{\rho_N(s,X^N)}{G_N}\bigr) \Bigr|^2  dX^N ds \\
& = 
 \sum_{i=1}^N \int_0^T\int_{\Pi^{dN}}
  G_N\,|\nabla_{x_i} \frac{\rho_N(s,X^N)}{G_N}|^2dX^N\,ds <+\infty 
\end{split}
\end{equation}
where $X^N= (x_1,\cdots,x_N)$.
This implies that $G_N\,\nabla_x \frac{\rho_N}{G_N}\in L^1([0,\ T]\times \Pi^{dN})$ and it allows to give meaning to Eq.~\eqref{liouvilleN} in the sense of distribution as it can be rewritten as
\begin{equation}\label{compactform}
\partial_t \rho_N-\sigma\,\mbox{div} \Big(G_N\,\nabla \frac{\rho_N}{G_N}\Big)=0.
\end{equation}
The proof of Prop. \ref{modulatedweak} follows the same path in our previous formal derivation. We of course start with the entropy control \eqref{entropycontrol} satisfied by the entropy solution which takes care of the terms in $dE_N/dt$ that are non-linear in $\rho_N$.

Then we recall that both $\bar\rho$ and $G_{\bar\rho}$ are smooth and non-vanishing. As we have previously seen, they satisfy
  \[
 \partial_t \frac{\bar\rho_N}{G_{\bar\rho_N}} -\frac{\sigma}{G_N}\,{\rm div} \Big(G_N \nabla \frac{\bar\rho_N}{G_{\bar\rho_N}}\Big)=R_N
   \] 
with $R_N$ given by \eqref{RN} and therefore satisfies
\begin{equation}
\begin{split}
 \partial_t \Big(\rho_N \log \frac{\bar\rho_N}{G_{\bar\rho_N}}\Big)
&  - \Big(\log \frac{\bar\rho_N}{G_{\bar\rho_N}}\Big) \partial_t \rho_N  \\
 &  -\sigma \frac{\rho_N G_{\bar\rho_N} }{G_N \bar\rho_N}\,{\rm div} \Big(G_N \nabla \frac{\bar\rho_N}{G_{\bar\rho_N}}\Big)=R_N \rho_N \frac{G_{\bar\rho_N}}{\bar\rho_N}.
\end{split}\label{limit}
\end{equation}
We note that all terms above are well defined in the sense of distribution with for example
\[
\begin{split}
&\frac{\rho_N G_{\bar\rho_N} }{G_N \bar\rho_N}\,{\rm div} \Big(G_N \nabla \frac{\bar\rho_N}{G_{\bar\rho_N}}\Big)={\rm div}\,\Big(\frac{\rho_N G_{\bar\rho_N} }{\bar\rho_N}\, \nabla \frac{\bar\rho_N}{G_{\bar\rho_N}}\Big)-\rho_N\,\nabla \frac{\bar\rho_N}{G_{\bar\rho_N}}\cdot \nabla \frac{G_{\bar\rho_N}}{\bar\rho_N}\\
&\qquad -G_N\, \nabla \frac{\rho_N}{G_N}\cdot\nabla \log \frac{\bar\rho_N}{G_{\bar\rho_N}}.
\end{split}
\]
Similarly the only non-smooth term in $R_N$ is
\[
-\sigma\,\sum_i \nabla_{x_i} \log G_N\cdot\nabla_{x_i} \frac{\bar\rho_N}{G_{\bar\rho_N}}=-\sigma\,\sum_i \frac{1}{G_N}\,\nabla_{x_i} G_N\cdot\nabla_{x_i} \frac{\bar\rho_N}{G_{\bar\rho_N}},
\]
which, once multiplied against $\rho_N\, \frac{G_{\bar\rho_N}}{\bar\rho_N}$, can be rewritten as
\[
-\sigma\,\sum_i \nabla_{x_i} \rho_N\cdot\nabla_{x_i} \log \frac{\bar\rho_N}{G_{\bar\rho_N}}+\sigma\,\sum_i G_N \nabla_{x_i} \frac{\rho_N}{G_N}\cdot\nabla_{x_i} \log \frac{\bar\rho_N}{G_{\bar\rho_N}}.
\]
It remains to use \eqref{compactform} and test it with $\log \bigl(\bar\rho_N/G_{\bar\rho_N}\bigr)$
to get after commuting space derivatives
\begin{equation}\nonumber
\begin{split}
 & \int_0^t\int - \Big(\log \frac{\bar\rho_N}{G_{\bar\rho_N}} \Big) \partial_t \rho_N   
  -\sigma  \int_0^t \int \frac{\rho_N G_{\bar\rho_N} }{G_N \bar\rho_N}\,{\rm div} \Big(G_N \nabla \frac{\bar\rho_N}  {G_{\bar\rho_N}}\Big)  \\ 
 & \hskip.5cm =  \sigma \bigg(- \int_0^t\int \rho_N |\nabla_x \log \frac{G_{\bar\rho_N}}{\bar\rho_N}|^2
 + 2 \int_0^t\int  \frac{G_N G_{\bar\rho_N}}{\bar\rho_N} \nabla \frac{\bar\rho_N}{G_{\bar\rho_N}}
 \cdot \nabla \frac{\rho_N}{G_N}
 \bigg).
\end{split} 
\end{equation}
Note that everything is well defined due to the control 
 $G_N\,\nabla_x \bigl({\rho_N}/{G_N}\bigr) \in L^1$.
It remains now to collect \eqref{entropycontrol}  with \eqref{limit} integrated in space and time on $(0,t)$
using identity \eqref{rN} to conclude the proposition.
\end{proof}

%%%%%%%%%%%%%%%%%%%%%%%
\subsection{Large range contribution: proof of Lemma \ref{regular}\label{secregular}}
%%%%%%%%%%%%%%%%%%%%%%%%%%%%%%%%%%%%%%%%%%%%%%%%%%%%%%
We mostly perform a direct calculation, with any special need to use the structure of the dynamics as before. First of all
  \[
  \begin{split}
    &\frac{1}{N}\int_{\Pi^{dN}} \rho_N \log \frac{G^W_{\bar \rho_N}}{G^W_N}\,dX^N==\frac{1}{2\,\sigma}\,\int\rho_N \int_{\Pi^{2d}} W(x-y)\,(d\mu_N-d\bar\rho)^{\otimes 2}\,dX^N\\
    &\ =\frac{1}{2\,N^2\,\sigma}\sum_{i,j} \int \rho_N\,\left(W(x_i-x_j)-W\star\bar\rho(x_i)-W\star\bar\rho(x_j)+\int \bar\rho\,W\star\bar \rho\right)\\
  \end{split}
  \]
  Therefore using the dynamics
  \[
  \begin{split}
    &\frac{d}{dt}\frac{1}{N}\int_{\Pi^{dN}} \rho_N \log \frac{W_{\bar \rho_N}}{W_N}\,dX^N\\
    &\ =\frac{1}{2\,N^2}\,\sum_{i,j}\int \rho_N \left(2\,\Delta W(x_i-x_j)-\Delta W\star\bar\rho(x_i)-\Delta W\star\bar \rho(x_j)\right)\\
    &\quad+\frac{1}{2\,N^2}\,\sum_{i,j}\int \rho_N \left(2\,\int \bar\rho\,\Delta W\star \bar\rho-\Delta W\star\bar\rho(x_i)-\Delta W\star\bar \rho(x_j)\right)\\
    &\quad-\frac{1}{N^3\,\sigma}\sum_{i,j,k} \int \rho_N\,\left(\nabla W(x_i-x_j)-\nabla W\star\bar\rho(x_i)\right)\cdot \nabla V(x_i-x_k)\\
    &\quad - \frac{1}{2\,N^2\,\sigma}\,\sum_{i,j} \int \rho_N\,\Big(2\,\int \nabla V\star \bar\rho\,\bar\rho\,\nabla W\star \bar\rho + \nabla W\star (\bar\rho\,\nabla V\star \bar\rho)(x_i)\\
    &\qquad\qquad-\nabla W\star (\bar\rho\,\nabla V\star \bar\rho)(x_j)\Big).
  \end{split}
  \]
  Now we just have to symmetrize as before, with first
 \[
 \begin{split}
&\frac{1}{2\,N^2}\,\sum_{i,j}\int \rho_N \left(2\,\Delta W(x_i-x_j)-\Delta W\star\bar\rho(x_i)-\Delta W\star\bar \rho(x_j)\right)\\
   &\ +\frac{1}{2\,N^2}\,\sum_{i,j}\int \rho_N \left(2\,\int \bar\rho\,\Delta W\star \bar\rho-\Delta W\star\bar\rho(x_i)-\Delta W\star\bar \rho(x_j)\right)\\
   &\quad=\int \rho_N\,\int_{\Pi^{2d}} \Delta W(x-y)\,(d\mu_N-d\bar\rho)^{\otimes 2}\,dX^N.
 \end{split}
 \]
 Similarly
  \[
 \begin{split}
   &-\frac{1}{N^3\,\sigma}\sum_{i,j,k} \int \rho_N\, ( \nabla W(x_i-x_j)-\nabla W\star\bar\rho(x_i))\cdot \nabla V(x_i-x_k)\\
   &\quad=-\frac{1}{\sigma} \int \rho_N\,\int_{\Pi^{2d}} \int_{\Pi^d} \nabla W(z-x)\cdot\nabla V(z-y)\,\ud \mu(z) \,(d\mu_N-d\bar\rho)^{\otimes 2}\\
   &\quad\quad -\frac{1}{N^2\,\sigma}\sum_{i,j} \int \rho_N\,\left(\nabla W(x_i-x_j)-\nabla W\star\bar\rho(x_i)\right)\cdot \nabla V\star\bar\rho(x_i).
 \end{split}
 \]
 Since $\nabla W$ is odd, 
  \[
 \begin{split}
   & \frac{1}{N^2}\,\sum_{i,j} \int \rho_N\,\nabla W\star (\bar\rho\,\nabla V\star \bar\rho)(x_i)\\
   &\quad =\int \rho_N\,\int_{\Pi^d} \nabla W\star (\bar\rho\,\nabla V\star \bar\rho)(x)\,\mu_N(dx)\\
   &\quad=-\int \rho_N\,\int_{\Pi^d}  \bar\rho\,\nabla V\star \bar\rho(z)\,\nabla W\star\mu_N(z)\,dz.
\end{split}
 \]
 Hence one also has
  \[
 \begin{split}
   & -\frac{1}{N^2\,\sigma}\sum_{i,j} \int \rho_N\,\left(\nabla W(x_i-x_j)-\nabla W\star\bar\rho(x_i)\right)\cdot \nabla V\star\bar\rho(x_i)\\
&- \frac{1}{2\,N^2\,\sigma}\,\sum_{i,j} \int \rho_N\,\Big(2\,\int \nabla V\star \bar\rho\,\bar\rho\,\nabla W\star \bar\rho + \nabla W\star (\bar\rho\,\nabla V\star \bar\rho)(x_i)\\
&\qquad\qquad + \nabla W\star (\bar\rho\,\nabla V\star \bar\rho)(x_j)\Big)\\
   &\quad=\frac{1}{\sigma}\,\int \rho_N\int_{\Pi^d} \nabla V\star \bar \rho(z)\,\Big((\nabla W\star \bar\rho - \nabla W\star\mu_N)\,\mu_N(dz) - \nabla W\star\bar\rho\,\bar\rho\,dz\\
   &\qquad\qquad + \nabla W\star\mu_N\,\bar\rho\,dz\Big)\\
   &\quad=-\frac{1}{\sigma}\,\int \rho_N\int_{\Pi^d} \nabla V\star \bar \rho(z)\,(\nabla W\star\mu_N-\nabla W\star \bar\rho)\,(\mu_N(dz)-\bar\rho\,dz),
 \end{split}
 \]
 which by an symmetrization equals to 
   \[
    -\frac{1}{2\,\sigma}\,\int \rho_N\int_{\Pi^{2d}} \nabla W(x-y)\,(\nabla V\star \bar \rho(x)-\nabla V\star\bar\rho(y))\,(\mu_N-\bar\rho)^{\otimes 2},
 \]
 and summing up all terms concludes the proof.
 %%%%%%%%%%%%%%%%%%%%%%%%%%%%%%%%%%%%%%%%%%%%%%%%%%%%%
%%%%%%%%%%%%%%%%%%%%%%%%%%%%%%%%%%%%%%%%%%%%%%%%%%%%%%
\section{Large deviation type estimates\label{LDE}}
%%%%%%%%%%%%%%%%%%%%%%%%%%%%%%%%%%%%%%%%%%%%%%%%%%%%%%%
The main goal of this section is to prove Prop \ref{prop3.1}, namely to derive a large deviation inequality on the functional
\[
\gamma\,\int_{\Pi^{dN}} F(\mu_N)\,\rho_N \,dX^N,
\]
where we recall that
\[
F(\mu)= - \int_{\Pi^{2d}\cap\{x\neq y\}} V_0(x-y)\,( d \mu- d \bar\rho)^{\otimes2}(x, y),
\]
and for $V_0$ satisfying the assumptions of \eqref{Flogexample}.

Classical large deviations approaches typically attempt at the limit of Gibbs equilibrium, see \cite{Arous,Var} for example. In contrast, we only care here about bounds on corresponding quantities. The estimates are also made more delicate since  $V_0(x)$ is allowed to be singular at $x=0$, \eqref{Flogexample} only imposes $V_0(x)\geq \log |x|\,\chi(|x|/\eta)$, and  we want to treat the best possible constant $\gamma$: any $\gamma<d$.

The main strategy that we follow is hence to try to remove the singularity in $V_0$ by carving out those $|x-y|<\eps$ for some $\eps$. This process is obviously the most delicate part and is carried out in subsection~\ref{proofprop3.1}.

This does not completely resolve the estimate though as it is still necessary to obtain a good quantitative control on the remaining functional in terms of $\eps$. The derivation of such a control forces us to revisit more classical large deviation approaches in subsections \ref{entropyLDE} and \ref{regularLDE}. 
%%%%%%%%%%%%%%%%%%%%%%%%%%%%%%%%%%%%%%%%%%%%%%%%%%%%%%%%%
\subsection{An explicit basic large deviation estimate\label{entropyLDE}} 
%%%%%%%%%%%%%%%%%%%%%%%%%%%%%%%%%%%%%%%%%%%%%%%%%%%%%%%%
%
For any $M>0$, define a decomposition of the torus $\Pi^d$ into $M^d$ 
 disjoint hypercubes $C_k^M$, $k=1,\dots, M^d$, of size $1/M$. We then denote 
\[
L_M(x,y)=M^{d}\,\ind_{x,\;y\in C_k^M},\quad L_M[f](x)=\int_{\Pi^d} L_M(x,y)\,f(dy).
\]
Note that for  $M\to \infty$, $L_M$ is an approximation of the Dirac mass as for $f$ Lipschitz
\[
|f(x)-L_M[f](x)|\leq \int_{\Pi^d} L_M(x,y)\,|f(x)-f(y)|\,dy\leq  C \|\nabla f\|_{L^\infty}\,\frac{1}{M}.
\]
The kernel $L_M$ makes it relatively straightforward to use elementary combinatorics for large deviation purposes. In particular we can derive the following estimate that we will make use of later.
\begin{Prop}
  There exists a constant $C_d$ s.t. for any $\bar \rho\in L^1(\Pi^d)$, one has the exponential bound for $M^d\leq N/2$,
  \[\begin{split}
  C^{-M^{d}}\,\frac{N^{M^{d}/2-1/2}}{M^{d\,M^d - d/2}}\leq&\int_{\Pi^{d\,N}} e^{N\,\int_{\Pi^d}\mu_N(dx)\,\log \frac{L_M[\mu_N](x)}{\bar\rho(x)}}\,\bar\rho_N\,dx_1\dots dx_N\\
  &\qquad\qquad\leq C\, N^{M^d+1/2}. 
\end{split}  \]\label{largedeviationbasic}
\end{Prop}

\begin{remark} Note that the previous proposition will be used choosing $M$ in terms of the number of particles $N$ and the regularized parameter $\varepsilon$ to provide the quantitative large deviation type estimate given by Proposition~\ref{largedeviation}.
\end{remark} 

\begin{proof}[Proof of Prop. \ref{largedeviationbasic}]
  The first step is to simply reduce to the case $\bar\rho=1$ by observing that
  \[
  \begin{split}
    e^{N\,\int_{\Pi^d}\mu_N(dx)\,\log \frac{L_M[\mu_N](x)}{\bar\rho(x)}}\,\bar\rho_N&=\exp \bigg( \sum_i \log \frac{L_M[\mu_N](x_i)}{\bar\rho(x_i)}+\sum_i \log \bar\rho(x_i) \bigg) \\
    &=\exp\bigg(\sum_i \log L_M[\mu_N](x_i)\bigg),
\end{split}
  \]
  and therefore it is enough to bound
  \[
Z_{N,M}=\int_{\Pi^{d\,N}} e^{N\,\int_{\Pi^d}\mu_N(dx)\,\log L_M[\mu_N](x)}\,dx_1\dots dx_N.
\]
Of course, one may simply write
\[
\begin{split}
  Z_{N,M}&=\int_{\Pi^{d\,N}} e^{\sum_i \log L_M[\mu_N ](x_i)}\,dx_1\dots dx_N\\
  &=\int_{\Pi^{d\,N}} \Pi_i L_M[\mu_N ](x_i)\,dx_1\dots dx_N.
\end{split}
\]
On the other hand, if $x_i\in C_k^M$
\[
L_M[\mu_N ](x_i)=\frac{M^d}{N}\,\#\{j\,|\;x_j\in C_k^M\}.
\]
This leads us to define, for any given $i$ the unique index $k(x_i)$ s.t. $x_i\in C_{k(x_i)}^M$, and for a given $k$ the number $n_k=\#\{j\,|\;x_j\in C_k^M\}$.
This simply gives
\[
Z_{N,M}=\frac{M^{d\,N}}{N^N}\,\int_{\Pi^{d\,N}} \Pi_i n_{k(x_i)}\,dx_1\dots dx_N.
\]
We can of course reverse the process and first choose any decomposition
\[
n_1+\ldots+n_{M^d}=N,\quad 0\leq n_k\leq N,
\]
and then denote by $\Omega=\Omega(n_1,\ldots,n_{M^d})$ the subset of $\Pi^{d\,N}$ such that $n_k=\# \{i\,|\;x_i\in C_k^M \}$ for all $k=1, \cdots, M^d$. Hence
\[\begin{split}
Z_{N,M}&=\frac{M^{d\,N}}{N^N}\,\sum_{n_1+\ldots+n_{M^d}=N} \Pi_k n_{k}^{n_k}\,\int_{\Omega(n_1,\ldots,n_{M^d})} \,dx_1\dots dx_N\\
&=\frac{M^{d\,N}}{N^N}\,\sum_{n_1+\ldots+n_{M^d}=N} \Pi_k n_{k}^{n_k}\,|\Omega(n_1,\ldots,n_{M^d})|. 
\end{split}
\]
It is relatively straightforward to evaluate $|\Omega(n_1,\ldots,n_{M^d})|$. We may first consider the reduced set $\Omega^r(n_1,\ldots,n_{M^d})$ where we assign particles to hypercubes based on their rank: Simply put $x_1,\ldots,x_{n_1}$ anywhere in $C_1^M$, $x_{n_1+1},\ldots,x_{n_1+n_2}$ anywhere in $C_2^M$ and so on...
Trivially
\[
|\Omega^r(n_1,\ldots,n_{M^d})|=M^{-d\,N}.
\]
On the other hand, up to a permutation $\tau$ of the indices, if $(x_1,\ldots,x_N)\in \Omega(n_1,\ldots,n_{M^d})$ then $(x_{\tau(1)},\ldots,x_{\tau(N)})\in \Omega^r(n_1,\ldots,n_{M^d})$, implying that
\[
\begin{split}
  &|\Omega(n_1,\ldots,n_{M^d})|\\
  &\quad =M^{-d\,N}\,\left|\left\{(k_1,\ldots k_N)\in \{1,\ldots,M^d\}^N\,|\;\forall l\ n_l=|\{i,\;k_i=l\}|\right\}\right|.
\end{split}
\]
We now recall the classical combinatorics results (see \cite{JaWa1} for instance) stating that
\begin{equation}
\left|\left\{(k_1,\ldots k_N)\in \{1,\ldots,M^d\}^N\,|\;\forall l\ n_l=|\{i,\;k_i=l\}|\right\}\right|=\frac{N!}{n_1!\dots n_{M^d}!}.\label{classiccomb}
\end{equation}
This proves that
\[
|\Omega(n_1,\ldots,n_{M^d})|=\frac{M^{-d\,N}\,N!}{n_1!\dots n_{M^d}!},
\]
and therefore
\begin{equation}
Z_{N,M}=\frac{N!}{N^N}\,\sum_{n_1+\ldots+n_{M^d}=N} \Pi_k \frac{n_k^{n_k}}{n_k!}. \label{basicZNinter}
\end{equation}
We may easily simplify further the expression by recalling as well that for any $n\geq 0$
\[
\sqrt{n+1}\,\left(\frac{n}{e}\right)^n\leq n!\leq  C\,\sqrt{n}\,\left(\frac{n}{e}\right)^n.
\]
Hence
\[
\frac{n_k^{n_k}}{n_k!}\leq \frac{1}{\sqrt{n_k +1}}\,e^{n_k},\quad \frac{N!}{N^N}\leq C\,N^{1/2}\,e^{-N},
\]
and
\[
Z_{N,M}\leq C\,N^{1/2}\,\sum_{n_1+\ldots+n_{M^d}=N} \Pi_k \frac{1}{\sqrt{n_k+1}}.
\]
Observe that we still have not lost much and we can also derive the lower bound
\[
Z_{N,M}\geq C^{-M^{d}}\,N^{1/2}\,\sum_{n_1+\ldots+n_{M^d}=N} \Pi_k \frac{1}{\sqrt{n_k+1}}.
\]
Here however for the upper bound, we proceed more roughly by simply bounding
\begin{equation}
Z_{N,M}\leq C\,N^{1/2}\,|\{(n_1,\ldots,n_{M^d})\in \mathbb{N}^{M^d}\,|\ \forall k\ n_k\geq 0,\ n_1+\ldots+n_{M^d}=N\}|.\label{basicZNinter2}
\end{equation}
Similarly, for the lower bound, we simply use the trivial estimate $(n_k+1)^{-1/2}\geq N^{-1/2}$ to find
\begin{equation}
  \begin{split}
    &Z_{N,M}\geq C^{-M^d}\,N^{1/2-M^d/2}\\
    &\qquad|\{(n_1,\ldots,n_{M^d})\in \mathbb{N}^{M^d}\,|\ \forall k\ n_k\geq 0,\ n_1+\ldots+n_{M^d}=N\}|.\label{basicZNinter3}
\end{split}
  \end{equation}

%{\bf A slight improvement here!  Using a basic inequality, one can find that  given that $n_1 + \cdots + %n_{M^d} = N$, 
%\[
%\Pi_{i=1}^{M^d}  \frac{1}{\sqrt{n_k + 1}} \geq \Big( \frac{N + M^d}{M^d} \Big)^{- M^d/2}
%\]   
% }  
  
We finally recall that (Lemma 7 in \cite{JaWa1})
\begin{equation}
|\{(b_1,\ldots,b_p)\in \mathbb{N}^{p}\,|\ \forall k\ b_k \geq  1,\  b_1+\ldots+b_{p}= q\}|=\binom{q-1}{p-1}.\label{classiccomb2}
\end{equation}
Defining $b_l=n_l+1$, $p=M^d$ and $q=N+M^d$, we obtain that
\[
\begin{split}
  &|\{(n_1,\ldots,n_{M^d})\in \mathbb{N}^{M^d}\,|\ \forall k\ n_k\geq 0,\ \mbox{and}\ n_1+\ldots+n_{M^d}=N\}|\\
  &\quad=\binom{N+M^d-1}{M^d-1},
\end{split}
\]
or from \eqref{basicZNinter2} and \eqref{basicZNinter3}
\[
C^{-M^d}\,N^{1/2-M^d/2}\,\binom{N+M^d-1}{M^d-1}\leq Z_{N,M}\leq C\,N^{1/2}\,\binom{N+M^d-1}{M^d-1}.
\]
It now only remains to trivially bound again the binomial coefficient. Denoting $K=M^d-1$ for simplicity and if $K\leq N/2$
\[
\begin{split}
  &\binom{N+M^d-1}{M^d-1}=\frac{(N+K)!}{K!\,N!}\leq C\,\frac{(N+K)^K}{K^K}\,\frac{(N+K)^{1/2}}{N^{1/2}\,K^{1/2}}\,\frac{(N+K)^N}{N^N}\\
  &\quad\leq C\,N^K\,e^{N\,\log(1+K/N)}\,\frac{2^K}{K^K}\leq C\,N^K\,\frac{2^K\,e^K}{K^K}\leq C\,N^K.
\end{split}
\]
as $\log(1+K/N)\leq K/N$. We similarly have the lower bound
\[
\begin{split}
  &\binom{N+M^d-1}{M^d-1}=\frac{(N+K)!}{K!\,N!}\geq \frac{1}{C}\,\frac{(N+K)^K}{K^K}\,\frac{(N+K)^{1/2}}{N^{1/2}\,K^{1/2}}\,\frac{(N+K)^N}{N^N}\\
  &\quad\geq \frac{1}{C}\,\frac{N^K}{K^{1/2}}\,e^{N\,\log(1+K/N)}\,\frac{1}{K^K}\geq \frac{1}{C}\,N^K\,\frac{1}{K^{K+1/2}}\geq \frac{1}{C}\,\frac{N^K}{K^{1/2+K}},
\end{split}
\]
which concludes the estimate.
  \end{proof}

%%%%%%%%%%%%%%%%%%%%%%%%%%%%%%%%%%%%%
\subsection{Quantitative large deviations and regularization\label{regularLDE}}
%%%%%%%%%%%%%%%%%%%%%%%%%%%%%%%%%%%%%
We can now write a quantitative large deviations for abstract functional. Consider any $F:\;\mathcal{P}(\Pi^d)\to \R$, possibly unbounded.
%but satisfying the following bound
%\begin{equation}
%\int_{\Pi^{dN}} |F(\mu_N)-F(L_\eps\star \mu_N)|\,\rho_N\,dX^N\leq C\,\eps^k+\frac{\lambda}{N}\int_{\Pi^{dN}} \rho_N\,\log \frac{\rho_N}{\bar \rho_N },\label{regularityF}
%\end{equation}
 Our goal in this subsection is to derive an intermediary large deviation inequality where we have removed the singularity in $F$. More specifically, we will derive an estimate on
\[
\frac{1}{N}\log \int_{\Pi^{dN}} e^{N\,F(L_\eps\star\mu_N)}\,\bar\rho_N\,dX^N,
\]
for some classical, smooth convolution kernel $L$, instead of the original
\[
\frac{1}{N}\log \int_{\Pi^{dN}} e^{N\,F(\mu_N)}\,\bar\rho_N\,dX^N.
\]
When $F$ is given by \eqref{Flogexample}, this will give the final bound once the singularity of $V_0$ at $0$ has been removed.
This bound is directly connected to the large deviation functional associated to $F$, of which we recall the definition
\begin{equation}
I(F)=\max_{\mu\in \mathcal{P}(\Pi^d)}  \Bigl[F(\mu)-\int_{\Pi^d} \mu\,\log \frac{\mu}{\bar\rho}\,dx\Bigr]. \label{largedeviationfunct}
  \end{equation}
We may now state the quantitative large deviation type  estimate for the regularized $F$
\begin{Prop}
  Assume that $\log\bar\rho\in W^{1,\infty}$, then there exists a constant $C$ depending only on $d$, $L$,  s.t. for any $F$, one has that
  \[\begin{split}
  &\frac{1}{N}\log \int_{\Pi^{dN}} e^{N\,F(L_\eps\star\mu_N)}\,\bar\rho_N\,dX^N\leq I(F)\\
  &\qquad+\frac{C}{N^{1/(d+1)}\,\eps^{d/(d+1)}}\,(\log N+|\log \eps|+\|\log\bar\rho\|_{L^\infty})+C\,\eps\,\|\log\bar \rho\|_{W^{1,\infty}}.
  \end{split}
  \]
  %Consequently, if $F$ satisfies \eqref{regularityF}, one has for any $\rho_N\in { \mathcal{P}}(\Pi^{dN})$ that
  %  \[\begin{split}  &  \int_{\Pi^{dN}} F(\mu_N)\,\rho_N\,dX^N\leq I(F)+(\lambda+1)\, \mathcal{H}_N(\rho_N\,\vert \;\bar\rho_N)\\ &\qquad+C\,N^{-\frac{k}{kd+d+1}}\,\left(1+\|\log\bar\rho\|_{W^{1,\infty}} + \mathcal{H}_N(\rho_N\,\vert \;\bar\rho_N)+\log N\right).  \end{split}  \]
  \label{largedeviation}
\end{Prop}
 
 The main idea to prove Prop. \ref{largedeviation} is to connect the left-hand side
 to an hypercubes averaging quantity for which it is  possible to apply the large deviation estimate in Prop.~\ref{largedeviationbasic} that we have proved previously.
   The first step is hence to replace the general $L_\eps\star \mu_N$ first by $L_\eps \star L_M [\mu_N]$ and then by $L_M[\mu_N]$ for the kernel $L_M$ previously defined.
First we observe that one has the following lemma
  \begin{lemma}
    There exists $C$ depending only on $d$ and $L$ s.t. for any measure $\mu\in{\mathcal{P}}(\Pi^d)$ and for $M\,\eps\geq 1$
    \[
\left\|L_\eps\star\mu-L_M[L_\eps\star\mu]\right\|_{L^1(\Pi^d)}\leq \frac{C}{M\,\eps}.
\]\label{changeconvolution}
  \end{lemma}
  \begin{proof}[Proof of Lemma \ref{changeconvolution}]
    The proof is rather straightforward and consists in noticing for example that for any $x$, if $L$ has support in $B(0,r)$ then
    \[\begin{split}
    &L_\eps\star\mu(x)=\sum_{k} \int_{C_k^M} L_\eps(x-y)\,\mu(dy)\\
    &\ =\sum_{k} \frac{1}{|C_k^M|}\,\int_{C_k^M} L_\eps(x-z)\,dz\,\int_{C_k^M} \mu(dy)\\
    &\quad+\sum_{k} \int_{C_k^M} \frac{1}{|C_k^M|}\,\int_{C_k^M} (L_\eps(x-y)-L_\eps(x-z))\,dz\,\mu(dy),
    \end{split}
    \]
    so that
    \[
    |L_\eps\star\mu(x)-L_\eps\star L_M[\mu](x)|\leq C\,\|L\|_{W^{1,\infty}}\,\frac{1}{M\,\eps}\, \sum_{k} \int_{C_k^M}\frac{\ind_{|x-y|\leq r\,\eps+1/M}}{\eps^d}\,\mu(dy),
    \]
    with 
    \[
\sum_{k} \int_{C_k^M}\frac{\ind_{|x-y|\leq r\,\eps+1/M}}{\eps^d}\,\mu(dy)\leq 2,
\]
since $\mu$ is a probability measure $\int d\mu=1$.
  \end{proof}
 
\medskip

\noindent Turning now back to the main proof.
\begin{proof}[Proof of Prop. \ref{largedeviation}]
We start by using the large deviation functional \eqref{largedeviationfunct} to find
 \[
  \begin{split}
    &\frac{1}{N}\log \int_{\Pi^{dN}} e^{N\,F(L_\eps\star\mu_N)}\,\bar\rho_N\,dX^N\\
    &\quad\leq \frac{1}{N}\log \int_{\Pi^{dN}} e^{N\,I(F)+N\,\int L_\eps\star\mu_N\,\log \frac{L_\eps\star \mu_N}{\bar\rho}}\,\bar\rho_N\,dX^N\\
    &\quad\leq I_F+\frac{1}{N}\,\log\int_{\Pi^{dN}} e^{N\,\int L_\eps\star\mu_N\,\log \frac{L_\eps\star \mu_N}{\bar\rho}}\,\bar\rho_N\,dX^N=I_F+I_N^\eps.
  \end{split}
  \]
One then finds from Lemma~\ref{changeconvolution} that 
  \[\begin{split}
  &\int L_\eps\star\mu_N\,\log \frac{L_\eps\star\mu_N}{\bar\rho}\leq \int L_\eps\star L_M[\mu_N]\,\log \frac{L_\eps\star L_M[\mu_N]}{\bar\rho}\\
  &\ +\|L_\eps\star\mu_N-L_\eps\star L_M[\mu_N]\|_{L^1}\,(|\log \eps|+\|\log\bar\rho\|_{L^\infty}). 
  \end{split}
  \]
Therefore
  \[
  \begin{split}
    &  I_N^\eps\leq \frac{1}{N}\log\int_{\Pi^{dN}} e^{N\,\int L_\eps\star L_M[\mu_N]\,\log \frac{L_\eps\star L_M[\mu_N]}{\bar\rho}}\,\bar\rho_N\,dX^N\\
    &\qquad+\frac{C}{M\,\eps}\,(|\log \eps|+\|\log\bar\rho\|_{L^\infty}).\\
  \end{split}
  \]
  Now we may just recall that $x \log x$ is a convex function so
  \[\begin{split}
  \int L_\eps\star L_M[\mu_N]\,\log L_\eps\star L_M[\mu_N]&\leq \int L_\eps\star \left(L_M[\mu_N]\,\log L_M[\mu_N]\right)\\
  &=\int L_M[\mu_N]\,\log L_M[\mu_N].
  \end{split}
  \]
  Moreover using the definition of $L_M$, it is straightforward to check that
  \[
\int L_M[\mu_N]\,\log L_M[\mu_N]=\int \mu_N\,\log L_M[\mu_N].
  \]
  Using now that $\log \bar\rho\in W^{1,\infty}$ and since $1/M\leq \eps$, we simply have that
  \[
-\int L_\eps\star L_M[\mu_N]\,\log\bar\rho\leq -\int \mu_N\,\log\bar\rho+C\,\eps\,\|\log\bar \rho\|_{W^{1,\infty}}.
\]
This leads to
\[
\begin{split}
    &  I_N^\eps\leq \frac{1}{N}\log\int_{\Pi^{dN}} e^{N\,\int \mu_N\,\log \frac{ L_M[\mu_N]}{\bar\rho}}\,\bar\rho_N\,dX^N\\
    &\qquad+\frac{C}{M\,\eps}\,(|\log \eps|+\|\log\bar\rho\|_{L^\infty})+C\,\eps\,\|\log\bar \rho\|_{W^{1,\infty}}.\\
\end{split}
\]
Using finally Prop. \ref{largedeviationbasic}, we find
\begin{equation}
I_N^\eps\leq C\,\frac{M^d}{N}\,\log N+\frac{C}{M\,\eps}\,(|\log \eps|+\|\log\bar\rho\|_{L^\infty})+C\,\eps\,\|\log\bar \rho\|_{W^{1,\infty}}.\label{INeps}
\end{equation}
It only remains to optimize in $M$ by choosing for example $M^{d+1}=N/\eps$ to conclude.  %For the second part, for any $\eps>0$,  we recall \eqref{regularityF}  
%  \[
%\int_{\Pi^{dN}} |F(\mu_N)-F(L_\eps\star \mu_N)|\,\rho_N\,dX^N\leq C\,\eps^k  +\lambda\,H_N(\rho_N\,\vert \;\bar\rho_N) .
%\]
%and use again  Lemma 1 from Jabin-Wang \cite{JaWa1}
%\[
%\int_{\Pi^{dN}} F(L_\eps\star\mu_N)\,\rho_N\,dX^N\leq H_N(\rho_N\,\vert \;\bar\rho_N)+\frac{1}{N}\,\log \int_{\Pi^{dN}} e^{N\,F(L_\eps\star\mu_N)}\,\bar\rho_N\,dX^N,
%\]
%to which we may now apply the first part of the theorem. This now gives us
%\[
%\begin{split}
%  &\int_{\Pi^{dN}} F(\mu_N)\,\rho_N\,dX^N\leq I(F)+H_N(\rho_N\,\vert \;\bar\rho_N) 
%    +C\,\eps^k  +\lambda\,H_N(\rho_N\,\vert \;\bar\rho_N) \\
%  &\ +\frac{C}{N^{1/(d+1)}\,\eps^{d/(d+1)}}\,(\log N+|\log \eps|+\|\log\bar\rho\|_{L^\infty})+C\,\eps\,\|\log\bar\rho\|_{W^{1,\infty}},
%  \end{split}
%\]
% leading to a last minimimization in $\eps$. Since $k\leq 1$, we may just take $\eps^{k+d/(d+1)}=N^{-1/(d+1)}$ concluding the proof.
\end{proof}
%%%%%%%%%%%%%%%%%%%%%%%%%%%%%%%%%%%%%%%%%%
\subsection{Estimating the large deviation functional}
%%%%%%%%%%%%%%%%%%%%%%%%%%%%%%%%%%%%%%%%%%
Prop.~\ref{largedeviation} controls the regularized large deviation inequality in terms of the classical large deviation functional. Our next step is hence to estimate this functional  and to prove that for some potential size $I(F)=0$ namely 
\begin{lemma}
    For any $\bar\rho\in L^\infty$ and any $c<d$, there exists a truncation $\delta$ s.t.  for any $\tilde V$ with $\|\tilde V\|_{L^1}\leq \delta$ and $\tilde V(x)\geq c\,\log |x|$,  and the functional
    \[
F_{\tilde V}(\mu)=-\,\int_{\{x\neq y\}} \tilde V(x-y)\,(\mu(dx)-\bar\rho(x)\,dx)\,(\mu(dy)-\bar\rho(y)\,dy),
  \]
  one then has that $I(F_{\tilde V})=0$.\label{I(F)log} 
\end{lemma}
This lemma precisely explains  why we  decompose  the potential  $V$ into short-range $V_0$ and long range $W$ in the main proof.  We will indeed later apply the lemma to $\tilde V= \frac{\lambda}{2 \sigma} V_0$ and choose $\eta$ s.t. $\|\tilde V\|_{L^1}<\delta$. Otherwise the case of optimality in the logarithmic Hardy-Littlewood-Sobolev shows that $I(F)\neq 0$.

\bigskip

\begin{proof}[Proof of Lemma \ref{I(F)log}]
    We start by recalling the classical logarithmic Hardy-Littlewood-Sobolev inequality (See Theorem 1 in \cite{CL} for instance), 
    \[
- \,\int_{\R^d} \log |x-y|\,\mu(dx)\,\mu(dy)   \leq  \frac 1 d \int_{\R^d} \mu\,\log \mu\,dx +  C_d,
\]
for some constant $C_d$ depending only on $d$ and any probability measure $\mu$.
We refer for instance to Dolbeault-Campos \cite{DoSe}  for a discussion of the importance of this inequality for the Patlak-Keller-Segel system.

This inequality shows that for $c<d$, using that $\tilde V\geq -c\,\log |x|$
\[\begin{split}
&F_{\tilde V}(\mu)-\int \mu\,\log \frac{\mu}{\bar\rho}\,dx\leq C_d\,( (1+\|\tilde V\|_{L^1})\,\|\bar\rho\|_{L^\infty} + \|\log \bar \rho \|_{L^\infty})\\
&\qquad-(1- \frac{c}{d} )\,\int \mu\,\log \frac{\mu}{\bar\rho}\,dx,
\end{split}
\]
and is hence coercive. This implies that, if we consider a maximizing sequence $\mu_n$, then $\mu_n$ is bounded in $L\,\log L$ and any weak limit is a maximum for $F_{\tilde V}(\mu)-\int \mu\,\log \frac{\mu}{\bar\rho}\,dx$.

The value of $I(F_{\tilde V})$ is hence given by a such a maximal measure $\bar\mu$. By standard arguments, a maximum $\bar\mu$ must satisfy that
on the support of $\bar\mu$
\[
1+\log \frac{\bar\mu}{\bar\rho} + {2}\,\tilde V\star(\bar\mu-\bar\rho)=\kappa.
\]
%with $K_\eta=\log |x-y|\,\ind_{|x-y|\leq\eta}$.
The constant $\kappa$ is chosen so that $\int \bar\mu=1$. Note that this implies that $\bar\mu$ cannot vanish on the support of $\bar\rho$ and hence
\[
\bar\mu=\frac{\bar\rho}{M}\,e^{ - {2}\,\tilde V\star(\bar\mu-\bar\rho)},\quad M=\int \bar\rho\,e^{ - {2\,}\tilde V\star(\bar\mu-\bar\rho)}\,dx.
\]
Let us denote $u= - \tilde V\star (\bar\mu-\bar\rho)$ and to emphasize the dependence on $u$ in $M$
\[
M=M_u=\int \bar\rho\,e^{{2}\,u(x)}\,dx.
\]
We observe that $u$ is a solution to
\begin{equation}
  u= - \tilde V\star\left(\bar\rho\,\left(\frac{e^{{2}\,u(x)}}{M_u}-1\right)
  \right),\label{eulerlagrange}
  \end{equation}
which is in fact a sort of non-linear elliptic equation. Our goal is simply to show that the unique solution to \eqref{eulerlagrange} is $u=0$ provided that $\delta$ is chosen small enough.

This is straightforward enough: First note that since $\bar\mu\in L\,\log L$ then $u\in L^\infty$ (by Lemma 1 in \cite{JaWa2} for instance) and in fact there exists $C$ depending only on $1-c/d$, $\|\tilde V\|_{L^1}$ and $\|\bar\rho\|_{L^\infty}$ s.t. $\|u\|_{L^\infty}\leq C$. Therefore
\[
\left|e^{{2 }\,u(x)}- e^{{2}\,u(y)}\right|\leq C\,|u(x) - u(y) | \leq C (|u(x)| + |u(y)|  )
\]
and
\[
 e^{-C}   \leq  M_u = \int \bar \rho(x) e^{{2}\,u(x) } d x   \leq e^C. 
\]  
Hence
\[
\begin{split} 
& \| \bar \mu - \bar \rho\|_{L^1} = \left\|\bar\rho\,\left(\frac{e^{{2}\,u(x)}}{M_u}-1\right)\right\|_{L^1}  \\
& \leq \frac{1}{M_u} \int \int \bar \rho(x) \bar \rho(y) |e^{{2}\, u(x) } - e^{{2 }\, u(y) }| \ud x \ud y \leq C\,\|u\|_{L^1},
\end{split} 
\]
for some constant $C$. To conclude, we note using \eqref{eulerlagrange} that
\[
\|u\|_{L^1}\leq C\,\|\tilde V\|_{L^1}\,\|u\|_{L^1},
\]
and it is enough to take $\delta$ small enough s.t. $C\,\|\tilde V\|_{L^1}<1$ to have that $u=0$ and finally $I(F_\eta)=0$.
  \end{proof}
%
%%%%%%%%%%%%%%%%%%%%%%%%%%%%%%%%%%%%%%%%%%%%%%%%%%%%
\subsection{Proof of Prop \ref{prop3.1}\label{proofprop3.1}}
%%%%%%%%%%%%%%%%%%%%%%%%%%%%%%%%%%%%%%%%%%%%%%%%%%%%
  We are finally ready to prove Prop. \ref{prop3.1}. We start by using again Lemma~\ref{duality} (following again Lemma 1 in \cite{JaWa2}) to obtain that
  \[
\gamma\,\int_{\Pi^{dN}} F(\mu_N)\,\rho_N\,dX^N\leq \mathcal{H}_N(\rho_N\,\vert\;\bar\rho_N)+\frac{1}{N}\log \int \bar\rho_N\,e^{N\,\gamma\,F(\mu_N)}\,dX^N,
  \]
  so that the whole question resolves around estimating
  \begin{equation}
Z_N(\gamma)=\int_{\Pi^{dN}} \bar\rho_N\,e^{N\,\gamma\,F(\mu_N)}\,dX^N.\label{partitionfunction}
    \end{equation}
 Of course if one studies the maximization problem
  \[
\sup_{\rho_N}  \bigg(  \gamma\,\int_{\Pi^{dN}} F(\mu_N)\,\rho_N\,dX^N- \mathcal{H}_N(\rho_N\,\vert\;\bar\rho_N) \bigg),
  \]
then the maximum is actually given by 
 \[
\rho_N=\frac{1}{Z_N}\, \bar \rho_N\,e^{N\,\gamma\,F(\mu_N)},\qquad Z_N=\int_{\Pi^{dN}} \bar \rho_N\,e^{N\,\gamma\,F(\mu_N)},
  \]
  and inserting this in the maximization problem exactly leads to $\frac{1}{N}\log Z_N$ with $Z_N$ given by \eqref{partitionfunction}. 

  \bigskip
  
$\bullet$ {\em Step~1: Introducing a regularized $F_\eps$.}   To estimate \eqref{partitionfunction}, we first introduce the regularized quantity
  \[
F_\eps(\mu)=-\int_{ \Pi^{2d} \cap \{x\neq y\}} V_\eps(x-y)\,(d\mu-d\bar\rho)^{\otimes2},
\]
with $V_\eps$  some regularized $V_\eps$. We denote
\[
Z_{N,\eps}(\gamma)=\int_{\Pi^{dN}} \bar\rho_N\,e^{N\,\gamma\,F_\eps(\mu_N)}\,dX^N,
\]
and the main point is to bound $Z_N(\gamma)$ in terms of the regularized $Z_{N,\eps}(\gamma')$ for some $\gamma'>\gamma$. The control on $Z_{N,\eps}$ will be performed at the end of the proof and follows in a straightforward manner from Prop.~\ref{largedeviation} and Lemma~\ref{I(F)log}.
To define $V_\eps$, we decompose $V_0$ by writing
\[
V_0 = \tilde V + \bar V,
\]
where
  \[
  \quad \tilde V=V_0-\bar V \hbox{ with }  \bar V(x)=\log |x|\,\chi(|x|/\eta),
  \]
  so that in particular $\tilde V\geq 0$ and still satisfies
  \begin{equation}
|\nabla \tilde V(x)|\leq \frac{C}{|x|^k} \hbox{ for } k >1/2.\label{nablatildeV}
  \end{equation}
We now choose $V_\eps(x)=\tilde V_\eps(x) + \bar V_\eps(x)$ 
where
\[
\tilde V_\eps(x)=\tilde V(x)\,(1-\chi(|x|/\eps^{1/2k}))
\hbox{ with } \bar V_\eps(x)=\log (\max(|x|,\,\eps))\,\chi(|x|/\eta). %\mbox{with}\ v_\eps=\min_{|z|\geq \eps^{1/2k}} \tilde V(z).
\]
Remark that we truncate $\tilde V_\eps$ at a much larger scale than $\bar V_\eps$: $\eps^{1/2k}$ vs $\eps$.

\bigskip

$\bullet$ {\em Step~2: Identifying the close and singular interactions.} Our next step is to relate $Z_N$ with 
\begin{equation}
I=\int_{\Pi^{dN}} \bar\rho_N\,e^{N\,\gamma\,F_\eps(\mu_N)} e^{\gamma\,\sum_{j>1} \log\frac{\eps}{|x_1-x_j|}\,\ind_{|x_1-x_j|\leq\eps}}\,\ud X^N.
\label{ZNZNeps11}
\end{equation}
The integral in $I$ clearly separates the regularized $F_\eps$ from the singularity in $V_0$. Moreover it identifies one test particle, which we choose as particle $1$, and compare all singularities through this particle.

By developing in $Z_N$, we have that
\begin{equation}
\begin{split}
  &F(\mu_N)=-\int_{\Pi^{2d}\cap\{x\neq y\}} V(x-y) \,( d \mu_N - d \bar\rho)^{\otimes2}\\
  &\quad  =-\frac{1}{N^2} \sum_{ i \ne j } V(x_i - x_j)
   + 2\frac{1}{N}\,\sum_i V\star\bar\rho(x_i)\\
  &\qquad-\int_{\Pi^{2d}\cap\{x\neq y\}} V(x-y)\,\bar\rho(x)\,\bar\rho(y)\,dx\,dy,
  \end{split}\label{development}
\end{equation}
with a similar formula for $F_\eps(\mu_N)$. Observe that for any $x$,
\[\begin{split}
|V_\eps\star\bar\rho(x)-V\star\bar\rho(x)|&\leq \|V_\eps-V\|_{L^1}\,\|\bar\rho\|_{L^\infty}\leq \|V\,\ind_{|z|\leq \eps^{1/2k}}\|_{L^1}\,\|\bar\rho\|_{L^\infty}\\
&\leq C\,\eps^{1/2kp^*}\,\|\bar\rho\|_{L^\infty}.
\end{split}
\]
Further note that since $\tilde V\geq 0$, we have that $\tilde V(x)\geq \tilde V_\eps(x)$ so this directly implies that
\[
F(\mu_N)\leq F_\eps(\mu_N)+C\,\|\bar\rho\|_{L^\infty}\,\eps^{1/2kp^*}+\frac{1}{N^2}\,\sum_{i\neq j} \log\frac{\eps}{|x_i-x_j|}\,\ind_{|x_i-x_j|\leq\eps}.
\]
Hence we obtain that
\begin{equation}
\begin{split}
& Z_N  \leq e^{C\,N\,\eps^{1/2kp^*}\,\|\bar\rho\|_{L^\infty}} \\
& \hskip2cm  \int_{\Pi^{dN}} \bar\rho_N\,e^{N\,\gamma\,F_\eps(\mu_N)+\frac{\gamma}{N}\,\sum_{i\neq j} \log\frac{\eps}{|x_i-x_j|}\,\ind_{|x_i-x_j|\leq\eps}}\,\ud X^N.
\end{split} \label{ZNZNeps1}
\end{equation}
By general H\"older inequality
\[
\begin{split}
  & \int_{\Pi^{dN}} \bar\rho_N\,e^{N\,\gamma\,F_\eps(\mu_N)+\frac{\gamma}{N}\,\sum_{i\neq j} \log\frac{\eps}{|x_i-x_j|}\,\ind_{|x_i-x_j|\leq\eps}}\,\ud X^N\\
  &\ =\int_{\Pi^{dN}} \bar\rho_N\,e^{N\,\gamma\,F_\eps(\mu_N)}\Pi_{i=1}^N e^{\frac{\gamma}{N}\,\sum_{j\neq i} \log\frac{\eps}{|x_i-x_j|}\,\ind_{|x_i-x_j|\leq\eps}}\,\ud X^N\\
  &\ \leq \Pi_{i=1}^N \left(\int_{\Pi^{dN}} \bar\rho_N\,e^{N\,\gamma\,F_\eps(\mu_N)} e^{\gamma\,\sum_{j\neq i} \log\frac{\eps}{|x_i-x_j|}\,\ind_{|x_i-x_j|\leq\eps}}\,\ud X^N\right)^{1/N}.
\end{split}
\]
Using the symmetry of $\rho_N$, we can simply keep one of the factors, for example with $i=1$ yielding
\[
\begin{split}
  &\int_{\Pi^{dN}} \bar\rho_N\,e^{N\,\gamma\,F_\eps(\mu_N)+\gamma\,\sum_{i\neq j} \log\frac{\eps}{|x_i-x_j|}\,\ind_{|x_i-x_j|\leq\eps}}\,\ud X^N\\
  &\ \leq\int_{\Pi^{dN}} \bar\rho_N\,e^{N\,\gamma\,F_\eps(\mu_N)} e^{\gamma\,\sum_{j>1} \log\frac{\eps}{|x_1-x_j|}\,\ind_{|x_1-x_j|\leq\eps}}\,\ud X^N=I.
  \end{split}
\]
Combined with \eqref{ZNZNeps1}, this gives
\begin{equation}
Z_N\leq e^{C\,N\,\eps^{1/2kp^*}\,\|\bar\rho\|_{L^\infty}}\,I.
  \label{ZNZNeps2}
\end{equation}

\bigskip

$\bullet$ {\em Step~3: Introducing the functional with a ``frozen'' test particle.} Since we will take $\eps$ very small, it is natural to expect that there will only be a limited number of indices $j$ s.t. $|x_1-x_j|\leq \eps$. To make that precise, we introduce the number $n$ of such indices $j$. Up to permutations, we may also assume that those are $j=2, \cdots,  n+1$ and we decompose accordingly
\begin{equation}
\begin{split}
  &I=\int_{\Pi^{dN}} \bar\rho_N\,e^{N\,\gamma\,F_\eps(\mu_N)} e^{\gamma\,\sum_{j>1} \log\frac{\eps}{|x_1-x_j|}\,\ind_{|x_1-x_j|\leq\eps}}\,\ud X^N\\
 & =\sum_{n=0}^{N-1}\binom{N-1}{n}\int_{\Pi^d}\ud x_1\,\int_{|x_1-x_i|\leq \eps,\;\forall i=2\dots n+1}\ud x_2\dots\ud x_{n+1}\\
  &\int_{|x_1-x_j|>\eps,\;\forall j>n+1} \ud x_{n+2}\dots \ud x_{N}\,\bar\rho_N\,e^{N\,\gamma\,F_\eps(\mu_N)} e^{\gamma\,\sum_{i=2}^{n+1} \log\frac{\eps}{|x_1-x_i|}}.
  \end{split}\label{ZNZNeps3}
\end{equation}
In this expression, one should first observe that $F_\eps(\mu_N)$ mostly do not depend on $x_2, \cdots,  x_{n+1}$. For this, denote $\mu_N^{1,n}$ the empirical measure obtained by replacing all $x_2\dots x_{n+1}$ by $x_1$
\[
\mu_N^{1,n}=\frac{n+1}{N}\delta(x-x_1)+\frac{1}{N}\sum_{j>n+1} \delta(x-x_j),
\]
and denote accordingly
\[
F^{1,n}_\eps(\mu_N)=F_\eps(\mu_N^{1,n}).
\]
Now $F^{1,n}_\eps(\mu_N)$ does not depend on $x_2\dots x_{n+1}$ but since $x_i$ and $x_1$ are close if $i=2\dots n+1$, we still expect it to be close to $F_\eps(\mu_N)$. Our next steps aim at making this precise.

\bigskip

$\bullet$ {\em Step~4: Comparing the potential $V_\eps$ for close particles.} We derive here the following estimate for $i=2\dots n+1$
\begin{equation}
  |\bar V_\eps(x_i-x_j)-\bar V_\eps(x_1-x_j)|\leq  %\frac{C}{|\log \eps|}\,\bar V_\eps(x_1-x_j)
  C_d\,\eps^{1/2}+\left\{\begin{aligned} &\log 2+\frac{C}{\eta}\,\eps\ \mbox{if}\ |x_1-x_j|\leq \eps^{1/2},\\
  &2\,\eps^{1/2}+\frac{C}{\eta}\,\eps \quad\mbox{if}\ |x_1-x_j|\geq \eps^{1/2}.
\end{aligned}
  \label{deltaVeps}
\right.
\end{equation}
We first recall that $k>1/2$. Hence for $|x_i-x_1|\leq \eps$, we  have that if $|x_i-x_j|\leq ({\eps^{1/2k}})/{2}$, then $|x_1-x_j|\leq \eps^{1/2k}$ and then
\[
\tilde V_\eps(x_i-x_j)=\tilde V_\eps(x_1-x_j)=0.
\]
Otherwise we necessarily have $|x_i-x_j|,\;|x_1-x_j|\geq  ({\eps^{1/2k}}/{4})$
and recalling \eqref{nablatildeV}, we obtain
\begin{equation}
\begin{split}
|\tilde V_\eps(x_i-x_j)-\tilde V_\eps(x_1-x_j)|
& \leq C_d\,\eps\,\max_{|z|\geq \frac{\eps^{1/2k}}{2}} |\nabla \tilde V(z)|\\
& \leq C_d\,\frac{\eps}{\eps^{1/2}}=C_d\,\eps^{1/2}.\label{deltatildeVeps}
\end{split}
\end{equation}
On the other hand for $|x_1-x_i|\leq \eps$, observe that we always have that
\[
\frac{1}{2}\,\max(\eps,|x_1-x_j|)\leq \max(\eps,|x_i-x_j|)\leq 2\,\max(\eps,|x_1-x_j|).
\]
Indeed one has trivially
\[
\max(\eps,|x_i-x_j|)\leq \max(\eps,|x_1-x_j|)+\eps\leq 2\,\max(\eps,|x_1-x_j|),
\]
and of course
\[
\max(\eps,|x_1-x_j|)\leq \max(\eps,|x_i-x_j|)+\eps\leq 2\,\max(\eps,|x_i-x_j|). 
\]
This implies that
\[
\left|\log\frac{1}{\max(\eps,|x_i-x_j|)}- \log\frac{1}{\max(\eps,|x_1-x_j|)}\right|\leq \log 2.
\]
If $|x_1-x_j|\geq \eps^{1/2}$ then we can be more precise as
\[
|\log a-\log b|=\int_{[a,\;b]} \frac{dx}{x} \leq |a-b|\,\max(1/a,1/b).
\]
In this case of course we have that $|x_i-x_j|\geq |x_1-x_j|-\eps\geq \eps^{1/2}/2$ so
\[
\left|\log\frac{1}{\max(\eps,|x_i-x_j|)}- \log\frac{1}{\max(\eps,|x_1-x_j|)}\right|\leq 2\,\eps^{-1/2}\,|x_1-x_i|\leq 2\,\eps^{1/2}.
\]
To summarize one has that
\[\begin{aligned}
& \left|\log\frac{1}{\max(\eps,|x_i-x_j|)}- \log\frac{1}{\max(\eps,|x_1-x_j|)}\right|\\
&\qquad\leq \left\{\begin{aligned} %\frac{C}{|\log \eps|}\,\log\frac{1}{\max(\eps,|x_1-x_j|)},
&\log 2\quad\mbox{if}\ |x_1-x_j|\leq \eps^{1/2}\\
& 2\,\eps^{1/2}\quad\mbox{if}\ |x_1-x_j|\geq \eps^{1/2}.\\
\end{aligned}\right.
\end{aligned}
\]
%for some constant $C$ and provided that $|x_1-x_j|\leq 1/e$. % ({\bf I changed $1$ to $1/e$}. Here the control can be improved to $C\eps/\max\{ \eps, |x_1 - x_j| \}$. )
And since the truncation $\chi$ is smooth, this yields
\begin{equation}
  |\bar V_\eps(x_i-x_j)-\bar V_\eps(x_1-x_j)|\leq  %\frac{C}{|\log \eps|}\,\bar V_\eps(x_1-x_j)
  \left\{\begin{aligned} &\log 2+\frac{C}{\eta}\,\eps\ \mbox{if}\ |x_1-x_j|\leq \eps^{1/2},\\
  &2\,\eps^{1/2}+\frac{C}{\eta}\,\eps \quad\mbox{if}\ |x_1-x_j|\geq \eps^{1/2}.
\end{aligned}
  \label{deltabarVeps}
\right.
\end{equation}
%Finally if $j\leq n+1$ then $|x_i-x_j|,\;|x_1-x_j|\leq 2\,\eps$ so that
%\[
%|V_\eps(x_i-x_j)-V_\eps(0)|\leq |\bar V_\eps(x_i-x_j)-\bar V_\eps(0)|\leq \log 2.
%\]
Combining \eqref{deltabarVeps} with \eqref{deltatildeVeps} proves \eqref{deltaVeps}.

\bigskip

$\bullet$ {\em Step~5: Introducing the intermediary scale $\eps^{1/2}$.}
The inequality \eqref{deltaVeps} shows that those $j$ s.t. $|x_1-x_j|\leq \eps^{1/2}$ will be playing a different role from those $j$ s.t. $|x_1-x_j|> \eps^{1/2}$. This leads us to introduce $n_{1/2}$ the number of $j$ s.t. $\eps\leq |x_1-x_j|\leq \eps^{1/2}$. Using again developments such as \eqref{development}, we have that
\[
\begin{split}
  &|F_\eps^{1,n}(\mu_N)-F_\eps(\mu_N)|\leq C\,\eps\,\|\bar\rho\|_{L^\infty}\\
  &\ +\frac{1}{N^2 }\,\sum_{i=2}^{n+1}\sum_{j\geq n+2} |V_\eps(x_i-x_j)-V_\eps(x_1-x_j)|\\
  &\ +\frac{1}{N^2}\sum_{i=2}^{n+1}\sum_{j=2, j\neq i}^{n+1} |V_\eps(x_i-x_j)-V_\eps(0)|.
\end{split}
  \]
Therefore, one has%, for $\delta_\eps=\frac{C}{|\log \eps|}$,
\begin{equation}
F_\eps(\mu_N)\leq F_\eps^{1,n}(\mu_N)+C\,\eps\,(\|\bar\rho\|_{L^\infty}+ \eta^{-1})+C\,\eps^{1/2} +\frac{n^2+n\,n_{1/2}}{N^2}\,\log 2. \label{Fepsupper}
\end{equation}
Note that by symmetry between $x_1$ and $x_i$ in the above bounds, we also have the symmetric
\begin{equation}
F_\eps^{1,n}(\mu_N)\leq F_\eps(\mu_N)+C\,\eps\,(\|\bar\rho\|_{L^\infty}+\eta^{-1}) +C\,\eps^{1/2} +\frac{n^2+n\,n_{1/2}}{N^2}\,\log 2, \label{Fepslower}
\end{equation}
which we will use later. Going back to \eqref{ZNZNeps3} and using \eqref{Fepsupper}, we may again freely assume that the $j$ s.t. $\eps\leq |x_1-x_j|\leq \eps^{1/2}$ are those indices from $n+2$ to $n+n_{1/2}+1$. We find that
\begin{equation}\begin{split}
  &I=\int_{\Pi^{dN}} \bar\rho_N\,e^{N\,\gamma\,F_\eps(\mu_N)} e^{\gamma\,\sum_{j>1} \log\frac{\eps}{|x_1-x_j|}\,\ind_{|x_1-x_j|\leq\eps}}\,\ud X^N\\
&\leq e^{C\,N\,(\eps\,(\|\bar\rho\|_{L^\infty}+\eta^{-1})+\eps^{1/2})}\sum_{n=0}^{N-1}\binom{N-1}{n} \sum_{n_{1/2}=0}^{N-1-n} \binom{N-n-1}{n_{1/2}}\,2^{n+n_{1/2}}\,J_{n,n_{1/2}}\\
\end{split}\label{boundIJ}
\end{equation}
with
\[
\begin{split}
&J_{n,n_{1/2}}= \int_{\Pi^d}\ud x_1\int_{|x_1-x_i|\leq \eps,\;\forall i=2\dots n+1}\!\!\!\!\!\! \ud x_2\dots\ud x_{n+1}\\
& \int_{\eps<|x_1-x_j|\leq \eps^{1/2},\;\forall j=n+2\dots n+1+n_{1/2}}\!\!\!\!\!\! \ud x_{n+2}\dots\ud x_{n+n_{1/2}+1}\\
  & \int_{|x_1-x_k|>\eps^{1/2},\;\forall k>n+1+n_{1/2}} \!\!\!\!\!\!\!\ud x_{n+n_{1/2}+2}\cdots \ud x_{N}\,\bar\rho_N\,e^{N\,\gamma\,F_\eps^{1,n}(\mu_N)} e^{\gamma\,\sum_{i=2}^{n+1} \log\frac{\eps}{|x_1-x_i|}}.
\end{split}
\]

\bigskip

$\bullet$ {\em Step~6: Bounding $J_{n,n_{1/2}}$ back in terms of $F_\eps$.}
We first recall that $F_\eps^{1,n}(\mu_N)$ does not depend on $x_2, \cdots,  x_{n+1}$ since it only depends on $\mu_N^{1,n}$. Hence the integrals in $J_{n,n_{1/2}}$ nicely separate. Moreover since $\gamma<d$
\[
\begin{split}
  & \int_{|x_1-x_i|\leq \eps,\;\forall i=2\dots n+1}\!\!\!\!\!\! \ud x_2\dots\ud x_{n+1} e^{\gamma\,\sum_{i=2}^{n+1} \log\frac{\eps}{|x_1-x_j|}}\\
  &\ =\left(\int_{|x_1-y|\leq \eps} \frac{\eps^\gamma}{|x_1-y|^\gamma}\,dy\right)^n\\
 &\ =\bar C^n\,\eps^{dn}=\bar C^n\,\int_{|x_1-x_i|\leq \eps,\;\forall i=2\dots n+1}\!\!\!\!\!\! \ud x_2\dots\ud x_{n+1},
  \end{split}
\]
with $\bar C\sim \frac{1}{d-\gamma}$.  
Therefore, one may obtain that
\[\begin{split}
%  &\int_{\Pi^{dN}} \bar\rho_N\,e^{N\,\gamma\,F_\eps(\mu_N)} e^{\gamma\,\sum_{j>1} \log\frac{\eps}{|x_1-x_j|}\,\ind_{|x_1-x_j|\leq\eps}}\,\ud X^N\\
%&\ \leq
%e^{C\,N\,\eps\,(\|\bar\rho\|_{L^\infty}+\eta^{-1})+C\,N\,\eps^{1/2}}\,\sum_{n=0}^{N-1}\binom{N-1}{n} \sum_{n_{1/2}=0}^{N-1-n} \binom{N-n-1}{n_{1/2}}\,
&J_{n,n_{1/2}}\leq
\bar C^n\,\int_{\Pi^d}\ud x_1\int_{|x_1-x_i|\leq \eps,\;\forall i=2\dots n+1}\!\!\!\!\!\! \ud x_2\dots\ud x_{n+1}\\
&\ \int_{\eps<|x_1-x_j|\leq \eps^{1/2},\;\forall j=n+2\dots n+1+n_{1/2}}\!\!\!\!\!\! \ud x_{n+2}\dots\ud x_{n+n_{1/2}+1}\\
  &\quad\int_{|x_1-x_k|>\eps^{1/2},\;\forall k>n+1+n_{1/2}} \ud x_{n+2}\cdots \ud x_{N}\,\bar\rho_N\,e^{N\,\gamma\,F_\eps^{1,n}(\mu_N)}.
\end{split}
\]
Now that we have used the key property of $F_\eps^{1,n}$, it is more convenient to change it back to $F_\eps(\mu_N)$ by using the reverse inequality \eqref{Fepslower},
\begin{equation}
J_{n,n_{1/2}}\leq e^{C\,N\,\eps\,(\|\bar\rho\|_{L^\infty}+\eta^{-1})+C\,N\,\eps^{1/2}}\,2^{n+n_{1/2}}\,\bar C^n\,\bar J_{n,n_{1/2}},\label{boundJbarJ}
\end{equation}
where
\[\begin{split}
%  &\int_{\Pi^{dN}} \bar\rho_N\,e^{N\,\gamma\,F_\eps(\mu_N)} e^{\gamma\,\sum_{j>1} \log\frac{\eps}{|x_1-x_j|}\,\ind_{|x_1-x_j|\leq\eps}}\,\ud X^N\\
%&\ \leq e^{C\,N\,\eps\,(\|\bar\rho\|_{L^\infty}+\eta^{-1})+C\,N\,\eps^{1/2}}\,\sum_{n=0}^{N-1}\binom{N-1}{n} \sum_{n_{1/2}=0}^{N-1-n} \binom{N-n-1}{n_{1/2}}\,
&\bar J_{n,n_{1/2}}\leq \int_{\Pi^d}\ud x_1\int_{|x_1-x_i|\leq \eps,\;\forall i=2\dots n+1}\!\!\!\!\!\! \ud x_2\dots\ud x_{n+1}\\
&\ \int_{\eps<|x_1-x_j|\leq \eps^{1/2},\;\forall j=n+2\dots n+1+n_{1/2}}\!\!\!\!\!\! \ud x_{n+2}\dots\ud x_{n+n_{1/2}+1}\\
  &\quad\int_{|x_1-x_k|>\eps^{1/2},\;\forall k>n+1+n_{1/2}} \ud x_{n+n_{1/2}+2}\cdots \ud x_{N}\,\bar\rho_N\,e^{N\,\gamma\,F_\eps(\mu_N)}.
\end{split}
\]

\bigskip

$\bullet$ {\em Step~7:  The final bound on $I$: Reconstructing the full integral.}
We now wish to ``undo'' the decompositions performed at step~5 and earlier at step~3 where we introduced $n_{1/2}$ and $n$. In other words, using \eqref{boundIJ} and \eqref{boundJbarJ}, we aim at expressing
\[
\begin{split}
  I\leq  &e^{C\,N\,\eps\,(\|\bar\rho\|_{L^\infty}+\eta^{-1})+C\,N\,\eps^{1/2}}\\
  &\qquad\qquad\sum_{n=0}^{N-1}\binom{N-1}{n} \sum_{n_{1/2}=0}^{N-1-n} \binom{N-n-1}{n_{1/2}}\,4^{n+n_{1/2}}\,\bar C^n\,\bar J_{n,n_{1/2}},
\end{split}
\]
in terms of the full integral over $\bar\rho_N\,e^{N\,\gamma\,F_\eps(\mu_N)}$.
  Unfortunately we cannot directly reverse the decomposition because of the extra factor $\bar C^n\,4^{n+n_{1/2}}$ in the sums over $n$ and $n_{1/2}$.

  We do expect the probability of $n$ or $n_{1/2}$ being of order $N$ to be extremely small of course and this issue can be solved by performing a last H\"older estimate at exponent $\gamma'/\gamma$ for some~$\gamma'$. 
\[\begin{split}
&%\int_{\Pi^{dN}} \bar\rho_N\,e^{N\,\gamma\,F_\eps(\mu_N)} e^{\gamma\,\sum_{j>1} \log\frac{\eps}{|x_1-x_j|}\,\ind_{|x_1-x_j|\leq\eps}}\,\ud X^N
I\leq e^{C\,N\,\eps\,(\|\bar\rho\|_{L^\infty}+\eta^{-1})+C\,N\,\eps^{1/2}}\, R_1^{1-\gamma/\gamma'}\,R_2^{\gamma/\gamma'},\\
\end{split}
\]
where
\[
\begin{split}
  &R_1=\sum_{n=0}^{N-1}\binom{N-1}{n}\,\sum_{n_{1/2}=0}^{N-1-n} \binom{N-n-1}{n_{1/2}}\,4^{\Lambda(n+n_{1/2})} \bar C^{\Lambda\,n}\,\eps^{d\,n}\,\eps^{d\,n_{1/2}/2},\\
  \end{split}
\]
with $\Lambda=\left(1-\frac{\gamma}{\gamma'}\right)^{-1}$.
On the other hand, we have 
\[
\begin{split}
  &R_2= \sum_{n=0}^{N-1}\binom{N-1}{n}  \sum_{n_{1/2}=0}^{N-1-n} \binom{N-n-1}{n_{1/2}}\\
  &\quad\int_{\Pi^d}\ud x_1\int_{|x_1-x_i|\leq \eps,\;\forall i=2\dots n+1}\!\!\!\!\!\! \ud x_2\dots\ud x_{n+1}\\
&\quad \int_{\eps<|x_1-x_j|\leq \eps^{1/2},\;\forall j=n+2\dots n+1+n_{1/2}}\!\!\!\!\!\! \ud x_{n+2}\dots\ud x_{n+n_{1/2}+1}\\
&\qquad \int_{|x_1-x_k|>\eps^{1/2},\;\forall k>n+1} \ud x_{n+n_{1/2}+2}\cdots \ud x_{N}\,\bar\rho_N\,e^{N\,\gamma'\,F_\eps(\mu_N)}.
\end{split}
\]
We easily have that
\[
R_1\leq (1+2\,\eps)^N,
\]
 provided that $4^\Lambda\,\bar C^\Lambda\,\eps^d\leq \eps$ and $4^\Lambda\,\eps^{d/2}\leq \eps$. As for $R_2$, we may now easily reverse the decomposition implemented in steps~3 and 5 to find
\[\begin{split}
%&\sum_{n=0}^{N-1}\binom{N-1}{n} \sum_{n_{1/2}=0}^{N-1-n} \binom{N-n-1}{n_{1/2}}\int_{\Pi^d}\ud x_1\int_{|x_1-x_i|\leq \eps,\;\forall i=2\dots n+1}\!\!\!\!\!\! \ud x_2\cdots\ud x_{n+1}\\
%&\ \int_{\eps<|x_1-x_j|\leq \eps^{1/2},\;\forall j=n+2\dots n+1+n_{1/2}}\!\!\!\!\!\! \ud x_{n+2}\dots\ud x_{n+n_{1/2}+1}\\
%&\quad \int_{|x_1-x_k|>\eps^{1/2},\;\forall k>n+1} \ud x_{n+2}\cdots  \ud x_{N}\,\bar\rho_N\,e^{N\,\gamma'\,F_\eps(\mu_N)}\\
&R_2 =\int_{\Pi^{dN}} \bar\rho_N \,e^{N\,\gamma'\,F_\eps(\mu_N)}\,\ud X^N=Z_{N,\eps}(\gamma').
\end{split}
\]
Therefore we obtain that 
\begin{equation}
\begin{split}
  &I=\int_{\Pi^{dN}} \bar\rho_N\,e^{N\,\gamma\,F_\eps(\mu_N)} e^{\gamma\,\sum_{j>1} \log\frac{\eps}{|x_1-x_j|}\,\ind_{|x_1-x_j|\leq\eps}}\,\ud X^N\\
  &\ \leq e^{C\,N\,\eps\,(\|\bar\rho\|_{L^\infty}+\eta^{-1})+C\,\eps^{1/2}}\,(1+2\,\eps)^N\,(Z_{N,\eps}(\gamma'))^{\frac{\gamma}{\gamma'}}.
\end{split}\label{finalboundI}
\end{equation}

\bigskip

$\bullet$ {Final Step: Using Prop. \ref{largedeviation} and Lemma \ref{I(F)log}.}
Let us first gather all our estimates: By inserting \eqref{finalboundI} into  \eqref{ZNZNeps2}, we have proved so far that
\begin{equation}
Z_N(\gamma)\leq e^{C\,N\,\eps\,(\|\bar\rho\|_{L^\infty}+\eta^{-1})+C\,N\,\eps^{1/2kp^*}+C\,\eps^{1/2}}\,(1+2\,\eps)^N\,(Z_{N,\eps}(\gamma'))^{\frac{\gamma}{\gamma'}}.\label{ZNZNepsfinal}
  \end{equation}
It only remains to bound $Z_{N,\eps}(\gamma')$. Note that since $V_\eps$ is smooth, for any convolution kernel $L$, we have that
\[
|V_\eps-L_{\eps'}\star L_{\eps'}\star V_\eps|\leq C\,\frac{\eps'}{\eps},
\]
and consequently
\[
\frac{1}{N}\,\log Z_{N,\eps}(\gamma')\leq C\,\frac{\eps'}{\eps}+\frac{1}{N}\,\log \int_{\Pi^{dN}}\,\bar\rho_N\,e^{N\,\gamma\,F_\eps(L_{\eps'}\star \mu_N)}\,dX^N.
\]
The estimate is now straightforward thanks to the Prop.~\ref{largedeviation} which directly shows that
\[
\begin{split}
  &\frac{1}{N}\,\log Z_{N,\eps}(\gamma')\leq I(\gamma'\,F_\eps)+ \frac{C}{N^{\frac{1}{d+1}}\,{\eps'}^{\frac{d}{d+1}}}\,(\log N+|\log \eps'|+\|\log\bar\rho\|_{L^\infty})\\
  &\qquad+C\,\eps'\,\|\log\bar \rho\|_{W^{1,\infty}}+C\,\frac{\eps'}{\eps}.
\end{split}
\]
From \eqref{ZNZNepsfinal}, and since we will take $\eps'<\eps$, this implies that
\[
\begin{split}
  &\frac{1}{N}\,\log Z_{N}(\gamma)\leq I(\gamma'\,F_\eps) + \frac{C}{N^{1/(d+1)}\,{\eps'}^{d/(d+1)}}\,(\log N+|\log \eps'|+\|\log\bar\rho\|_{L^\infty})\\
  &\qquad +C\,\eps\,(\eta^{-1}+\|\log\bar \rho\|_{W^{1,\infty}} + \| \bar \rho\|_{L^\infty} )+C\,\eps^{1/2}+C\,N\,\eps^{1/2kp^*}+C\,\frac{\eps'}{\eps}.
\end{split}
\]
We may estimate $I(\gamma'\,F_\eps)$ through Lemma~\ref{I(F)log}. We observe that as long as $\gamma'<d$, we indeed have that $V_\eps(x)\geq c\,\log|x|$ for some $c<d$. On the other hand, since $V\in L^p$ and $\mbox{supp}\,V\in B(0,\eta)$, by choosing $\eta$ small enough, we can guarantee that $\|V\|_{L^1}\leq \delta$.
Now simply taking for example $\eps'=N^{-1/(2d+1)}$ and $\eps=\sqrt{\eps}$, we deduce that provided $\gamma'<d$, there exists some $\theta>0$ s.t.
\[\begin{split}
\frac{1}{N}\,\log Z_{N,\eps}(\gamma')\leq &\frac{C}{N^{1/(2d+1)}}\,(\log N+\|\log\bar\rho\|_{W^{1,\infty}}+\eta^{-1})\\
&+\frac{C}{N^\theta}.
\end{split}
\]
Note that of course $\eps>>N^{-1/d}$. On the other hand we need $N$ large enough so that $\eps^{d-1}=N^{-(d-1)/2\,(2d+1)}\leq \bar C^{-\Lambda}$ which gives the condition on $N$ in the proposition. 
%%%%%%%%%%%%%%%%%%%%%%%%%%%%%%%%%%%%%%
%%%%%%%%%%%%%%%%%%%%%%%%%%%%%%%%%%%%%%%%%%%%%% 
 \section{Appendix}
 %%%%%%%%%%%%%%%%%%%%%%%%%%%%%%%%%%%%%%%%%%
 %%%%%%%%%%%%%%%%%%%%%%%%%%%%%%%%%%%%%%%%%%%
 % \begin{remark} It is interesting to note the importance of the free energy to cancel the bad quantity which appeared in the defintion of entropy solution in  \cite{JaWa2} and asked for some restrictions on the interaction kernels.  \end{remark}
 %%%%%%%%%%%%%%%%%%%%%%%%%%%%%%%%%%%%%%%%%%%%%%%%%%%%%%%%%%%%%%%%%%%%%%%%%%%%%
\subsection{Large deviation type estimates in \cite{JaWa2}}
%%%%%%%%%%%%%%%%%%%%%%%%%%%%%%%%%%%%%%%%%%%%%%%%%%%%%%%%%%%%%%%%%%%%%%%%%%%%%%
We first recall and prove Lemma~\ref{duality} which was Lemma 1 in \cite{JaWa2} 
\begin{lemma}\label{duality2}
  For any $\rho_N,\;\bar\rho_N\in \mathcal{P}(\T^{dN})$, any test function $\psi \in L^\infty(\Pi^{dN})$, one has that for any  $\alpha>0$, 
  \[
\int_{\T^{dN}} \psi(X^N)\,d\rho_N\leq \frac{1}{\alpha} \,  \frac{1}{N}\int d\rho_N\,\log \frac{\rho_N}{\bar\rho_N}+ \frac{1}{\alpha} \frac{1}{N}\log \int_{\T^{dN}} e^{ \alpha   N\,\psi(X^N)}\,d\bar\rho_N. 
  \]
  \end{lemma}
\begin{proof}
Without loss of generality, we assume that $\alpha =1$.  Define
\[
f=\frac{1}{\lambda}\,e^{N\,\psi}\,\bar\rho_N,\quad \lambda=\int_{\Pi^{d\, N}} d\bar\rho_N\,e^{N\,\psi}.
\]
Notice that $f$ is a probability density as $f\geq 0$ and $\int f=1$. Hence by the convexity of the entropy
\[
\frac{1}{N}\int_{\Pi^{d \,N}} \rho_N\,\log f\,dX^N\leq \frac{1}{N}\int_{\Pi^{d\,N}} \rho_N\,\log \rho_N\,dX^N.
\]
Expanding the left-hand side 
\[\begin{split}
\frac{1}{N}\!\!\int_{\Pi^{d\,N}}\!\! \rho_N\,\log f\,dX^N&=\int_{\Pi^{d\,N}}\!\! \rho_N\,\Phi\,dX^N+\frac{1}{N}\int_{\Pi^{d\,N}}\!\! \rho_N\,\log \bar\rho_N\,d X^N-\frac{\log\lambda}{N},
\end{split}\]
gives the desired inequality.
  \end{proof}
Lemma \ref{duality2} directly connects bounds on quantities like $\int \psi(X^N) d \rho_N $ to the relative entropy $\mathcal{H}_N$ and estimates on quantities that can be seen as partition functions
\begin{equation}\displaystyle 
\int_{\T^{dN}} e^{\displaystyle 
 N\,\int_{\Pi^{2d}} f(x,y)\,(d\mu_N-d\bar\rho)^{\otimes 2}}\,\bar\rho^{\otimes N}\,dX^N. \label{partition}
  \end{equation}
It is hence natural to try to use large deviation type of tools to bound \eqref{partition}. Note however that our goals here are different from classical large deviation approaches: We do not try to calculate the limit as $N\to \infty$ of \eqref{partition} but instead to obtain bounds that are uniform in $N$.

We now recall the estimate from \cite{JaWa2} 
\begin{theorem}  {\rm (Theorem 4 in \cite{JaWa2})}.   Consider $\bar\rho \in L^1(\Pi^d) $ with $\bar\rho \geq 0$ and $\int_{\Pi^d} \bar\rho \ud x =1$. Consider further any $\phi(x,z)\in L^\infty$ with
\[
\gamma :=  C\,  \bigg( \sup_{p \geq 1} \frac{\|\sup_z |\phi(.,z)|\|_{L^p(\bar\rho \ud x)}}{p} \bigg)^2   <1,
\]  
where $C$ is a universal constant. Assume that $\phi$ satisfies the following cancellations
\begin{equation}\label{TwoCanLDP}
\int_{\Pi^d} \phi(x,z)\,\bar\rho(x)\,dx=0\quad\forall z, \qquad \int_{\Pi^d} \phi(x,z)\,\bar\rho(z)\,dz=0\quad\forall x.
\end{equation}
Then 
\begin{equation}
\label{ME2}
\int_{\Pi^{d\,N}} \bar{\rho}_N \exp\bigg(\frac{1}{N}\sum_{i,j=1}^N \phi(x_i,x_j)\bigg) \ud X^N \leq \frac{2}{1-\gamma} < \infty,
\end{equation}
where again $\bar \rho_N = \bar \rho^{\otimes N}$. \label{largedeviationJaWa}
\end{theorem}
We may directly deduce Theorem~\ref{largedeviationJaWamod} from this.  Given a  configuration $X^N=(x_1, \cdots, x_N)$ or $\mu_N = \frac{1}{N}\sum_{i=1}^N \delta_{x_i}$, just write
\[
N \int_{\Pi^{2d}} f(x, y) (d \mu_N - d \bar \rho)^{\otimes 2 } = \frac{1}{N} \sum_{i, j=1}^N \phi(x_i, x_j), 
\]
where 
\[
\begin{split} 
& \phi(x, y) = f(x, y) - \int_{\Pi^d} f(x, w) \bar \rho(w) d w  \\
&\qquad  - \int_{\Pi^d} f(z, y) \bar \rho(z) d z + \int_{\Pi^{2d}} f(z, w) \bar \rho(z) \bar \rho(w) d z d w.  
\end{split} 
\]
This new $\phi$ is a symmetrization of $f$ according to the reference  measure $\bar \rho$ and $\phi$ indeed satisfies two cancellation rules in Theorem~\ref{largedeviationJaWa}, i.e. 
\begin{equation}\label{mean0}
\int_{\Pi^d} \phi(x, y) \bar \rho(y) d y  = 0, \forall x, \quad \int_{\Pi^d } \phi(x, y) \bar \rho(x) d x = 0, \forall y. 
\end{equation} 
Finally $\|\phi\|_{L^\infty}\leq 4\,\|f\|_{L^\infty}$ so that we only need to take $\alpha$ small enough such that
\[
\gamma= C\,  \left( \sup_{p \geq 1} \frac{\|\sup_z |\alpha\,\phi(.,z)|\|_{L^p(\bar\rho \ud x)}}{p} \right)^2\leq 16\,C\,\alpha^2\,\|f\|_{L^\infty}^2<1.
\]
We also want to emphasize here that in the case $\phi\in L^\infty$, a  probabilistic proof of Theorem \ref{largedeviationJaWa} was recently obtained in \cite{LimLuNo}. 

 %%%%%%%%%%%%%%%%%%%%%%%%%%%%%%%%%%%%%%%%%%%%%%%%%%%%%%%%%%%%%%%%%%%%%%%%%
 \subsection{Existence of entropy solution for the Liouville equation with the Patlak-Keller-Segel interaction kernel in 2D}
%%%%%%%%%%%%%%%%%%%%%%%%%%%%%%%%%%%%%%%%%%%%%%%%%%%%%%%%%%%%%%%%%%%%%%%%%%%
For the reader's convenience, we prove here the existence of an entropy solution to Eq. \eqref{liouvilleN} in the case of the Patlak-Keller-Segel interaction kernel in dimension 2 namely
\begin{equation}
V(x) = \lambda \log |x| + V_e(x)\label{VKS}
\end{equation}
with $0<\lambda< 2 d \sigma $ and  $V_e$ a smooth correction so that $V$ is periodic.

In general obtaining well-posedness to \eqref{liouvilleN} may require a different set of assumptions than what we need to derive the mean field limit. In particular for existence as here, we need to be more specific than just asking $V(x)\geq \lambda\,\log|x|$ together with bounds on $|\nabla V|$. Moreover we emphasize that the argument below only shows existence of solutions to the Liouville equation ~\eqref{liouvilleN}. The existence of solutions to the original coupled SDE system~\eqref{sys} is much more difficult and essentially open out of the diffusion-dominated regime studied in \cite{FoJo}.
Here we prove
\begin{Prop}
  Assume that the initial data $\rho^0_N$ satisfies that
  \[
\int_{\Pi^{dN}} \rho^0_N  \log \Big(\frac{\rho^0_N}{G_{N,\eps}} \Big)<\infty.
\]
Then there exists a global in time  entropy solution to \eqref{liouvilleN} with $V$ given by~\eqref{VKS}.\label{propexistence}
\end{Prop}
\begin{proof}
We consider the following regularization of $V$
\[
\bar V_\varepsilon = \lambda \log \max(|x|,\varepsilon).
\]
Since $\bar V_\eps$ is now Lipschitz, we trivially have existence of a smooth solution $\rho_{N,\eps}$ to \eqref{liouvilleN} for this interaction kernel (it is a standard, linear advection-diffusion equation). The goal is of course to pass to the limit in $\rho_{N,\eps}$ with two difficulties: Handle the singular interaction terms $\nabla V_\eps(x_i-x_j)\,\rho_{N,\eps}$ and obtain the non-linear entropy bound at the limit.

The first step is to use Prop.~\ref{modulatedsmooth} for this corresponding solution $\rho_{N,\eps}$ (again this is straightforward since $\bar V_\eps$ is Lipschitz). This yields the entropy bound   
\begin{equation}
\begin{split}
&\int_{\Pi^{dN}} \rho_{N, \eps} (t,X^N) \log \bigl(\frac{\rho_{N,\eps} (t,X^N)}{G_{N,\eps} }\bigr) \, dX^N \\
& + \sigma  \sum_{i=1}^N\int_0^t \int_{\Pi^{dN}} 
    \rho_{N,\eps} (s,X^N) \Bigl|\nabla_{x_i} 
       \log \bigl(\frac{\rho_{N,\eps}(s,X^N)}{G_{N,\eps} }\bigr) \Bigr|^2  dX^N ds \\
&
     \le \int_{\Pi^{dN}} \rho^0_N  \log \bigl(\frac{\rho^0_N}{G_{N,\eps}}\bigr) \, dX^N. \label{entropycontroleps} 
\end{split}
\end{equation}
The next step is to use our large deviation estimates, namely Prop.~\ref{prop3.1}: Since $\lambda<2 d \sigma$, we have  that for some constant $C$ independent of $\eps$
   \begin{equation}
\int_{\Pi^{dN}} \rho_{N,\eps}\log G_{N,\eps}\leq C,\quad \int_{\Pi^{dN}} \rho_{N,\eps}\log \rho_{N,\eps}\leq C.\label{equi-integrability}
\end{equation}
This implies that $\rho_{N,\eps}$ is equi-integrable in $X^N$ and we may now extract a converging subsequence (still denoted by $\rho_{N,\eps}$) s.t. $\rho_{N,\eps}\to \rho_N$ weakly in $L^\infty([0,\ T],\ L^1(\Pi^{dN}))$.

Next we start to use the specific structure of $V_\eps$. Since $V\leq 0$,  $G_{N,\eps}$ is increasing in $\eps$ (when $\eps \to 0$). Moreover $G_{N,\eps}\geq 1$ so in particular $1/G_{N,\eps}$ is bounded in $L^\infty$ by $1$ and converges pointwise to $1/G_N$. It hence converges {\em strongly} in every space between $L^1$ and $L^\infty$ strictly.
Let us  now denote
   \[
X_\eps=\frac{\rho_{N,\eps}}{G_{N,\eps}}.
   \]
   We next observe that $X_\eps$ converges weakly to $X=\rho_N/G_N$ as $\eps\to 0$. This is a consequence of the equi-integrability of $\rho_{N,\eps}$ given by \eqref{equi-integrability} and the above strong convergence of $G_{N,\eps}$. To be more specific, fix any $\eps_0$, then for $\eps\leq \eps_0$, and any smooth test function $\phi$, 
   \[
\int \frac{\rho_{N,\eps}}{G_{N,\eps}}\,\phi=\int_{G_{N,\eps_0}\leq M} \frac{\rho_{N,\eps}}{G_{N,\eps}}\,\phi+\int_{G_{N,\eps_0}\geq M} \frac{\rho_{N,\eps}}{G_{N,\eps}}\,\phi,
\]
where we choose any $M$ s.t. $M< \exp(\frac{1}{N\,\sigma}\,\log \frac{1}{\eps_0})$. 
   As $G_{N,\eps}\geq 1$, we have that
   \[
\int_{G_{N,\eps_0}\geq M} \frac{\rho_{N,\eps}}{G_{N,\eps}}\,\phi\leq \int_{G_{N,\eps_0}\geq M} \rho_{N,\eps}\,|\phi|\leq C\,\frac{\|\phi\|_{L^\infty}}{\log M} 
   \]
  by the first point of \eqref{equi-integrability} since $\{G_{N,\eps_0}\geq M\}\subset \{G_{N,\eps}\geq M\}$. By doing the same estimate at the limit, we have that
   \[\begin{split}
&\left|\int \phi\left(\frac{\rho_{N,\eps}}{G_{N,\eps}}-\frac{\rho_{N}}{G_{N}}\right)\right|\leq C\,\frac{\|\phi\|_{L^\infty}}{\log M}+\left|\int_{G_{N,\eps_0}\leq M} \phi\left(\frac{\rho_{N,\eps}}{G_{N,\eps}}-\frac{\rho_{N}}{G_{N}}\right)\right|.
   \end{split}
   \]
   But now we note that $G_{N,\eps}$ is in fact uniformly smooth in $\eps$ on ${G_{N,\eps_0}\leq M}$. Indeed denoting by $\delta=\min_{i\neq j} |x_i-x_j|$, we have the trivial bound $G_{N,\eps_0}\geq \exp(-\frac{1}{N\,\sigma}\,\log \max(\delta,\eps_0))$. We recall that $M< \exp(\frac{1}{N\,\sigma}\,\log \frac{1}{\eps_0})$ which implies that $\exp(\frac{1}{N\,\sigma}\,\log \frac{1}{\delta})\leq M$ thus providing an explicit control on $\delta$ in terms of $M$. Hence on ${G_{N,\eps_0}\leq M}$, $G_{N,\eps}$ convergence smoothly to $G_{\eps}$. By the weak convergence of $\rho_{N,\eps}$, we have that
   \[
\limsup_{\eps\to0}\left|\int \phi\left(\frac{\rho_{N,\eps}}{G_{N,\eps}}-\frac{\rho_{N}}{G_{N}}\right)\right|\leq C\,\frac{\|\phi\|_{L^\infty}}{\log M}. 
   \]
   Taking now $\eps_0\to 0$ and hence $M\to\infty$, we conclude that $X_\eps\to X$.
   
   \medskip

   Now the entropy inequality \eqref{entropycontrol} at $\eps$ implies that for any $\eps_0\geq \eps$
   \begin{equation}
     \begin{split}
\int_{\Pi^{dN}} \rho_{N,\eps}\,\log \frac{\rho_{N,\eps}}{G_{N,\eps}}+\sigma \int_0^t\int_{\Pi^{dN}} \frac{|\nabla X_\eps|^2}{X_\eps}\,G_{N,\eps_0}\leq \int_{\Pi^{dN}} \rho_N^0\,\log \frac{\rho_N^0}{G_{N,\eps}},
\end{split}\label{entropycontrolmod}
   \end{equation}
   again because $G_{N,\eps_0}\leq G_{N}$.
   
Convexity and the large deviation estimates from Prop.~\ref{prop3.1} show that
\[
\int_{\Pi^{dN}} \rho_{N}\,\log \frac{\rho_{N}}{G_{N}}\leq \liminf_{\eps\to 0}
 \int_{\Pi^{dN}} \rho_{N,\eps}\,\log \frac{\rho_{N,\eps}}{G_{N,\eps}}.
\]
Since $G_{N,\eps}$ is increasing in $\eps$, we also have directly that
\[
\int_{\Pi^{dN}} \rho_{N}^0\,\log \frac{\rho_{N}^0}{G_{N,\eps}}\leq \int_{\Pi^{dN}} \rho_{N}^0\,\log \frac{\rho_{N}^0}{G_{N}}.
\]
It only remains to treat the dissipation term is in \eqref{entropycontrolmod}. Since $\eps_0$ is fixed then $G_{N,\eps_0}$ is now smooth, $\nabla X_\eps$ and $X_\eps$ converge both in the sense of distribution and, by convexity, we have that
\[
\int_0^t\int_{\Pi^{dN}} \frac{|\nabla X|^2}{X}\,G_{N,\eps_0}\leq \liminf_{\eps\to 0} \int_0^t\int_{\Pi^{dN}} \frac{|\nabla X_\eps|^2}{X_\eps}\,G_{N,\eps_0}.
\]
This gives that for any $\eps_0$,
\[
\int_{\Pi^{dN}} \rho_{N}\,\log \frac{\rho_{N}}{G_{N}}+\sigma \int_0^t\int_{\Pi^{dN}} \frac{|\nabla X|^2}{X}\,G_{N,\eps_0}\leq \int_{\Pi^{dN}} \rho_N^0\,\log \frac{\rho_N^0}{G_{N}}.
\]
We recall again that $G_{N,\eps}$ is increasing in $\eps$ so, by the monotone convergence theorem, this yields the desired entropy bound
\[
\int_{\Pi^{dN}} \rho_{N}\,\log \frac{\rho_{N}}{G_{N}}+\sigma \int_0^t\int_{\Pi^{dN}} \frac{|\nabla X|^2}{X}\,G_{N}\leq \int_{\Pi^{dN}} \rho_N^0\,\log \frac{\rho_N^0}{G_{N}}.
\]
\medskip

The last question is to show that Eq.~\eqref{liouvilleN} is satisfied in the sense of distribution at the limit.  The only difficulty is to pass to the 
the limit in the advection term, which we may rewrite in a non-linear form as
$${\rm div} \Big(\rho_{N,\varepsilon} \nabla \log \frac{\rho_{N,\varepsilon}}{G_{N,\varepsilon}}\Big) 
$$
To do so
let us introduce the following quantity
$$I= \int_{\Pi^{dN}} \varphi \rho_{N,\varepsilon} (\nabla \log G_{N,\varepsilon}  - \nabla \log \rho_{N,\varepsilon}) 
$$
where $\varphi$ is a ${\mathcal C}^\infty$ test function.
Let us define a cut-off function $\chi$ such that $\chi(|x|) = 1$ is $|x|\le 1$ and $\chi(|x|)= 0$
if $|x|>2$ and choose $\varepsilon$ and $M$ such that again
$$2\,M < \exp \bigl(\frac{C}{N} \log 1/\varepsilon_0 \bigr)$$
 then
\begin{eqnarray}
 I  = 
 &&  \int_{\Pi^{dN}} \varphi \rho_{N,\varepsilon} (\nabla \log G_{N,\varepsilon}  - \nabla \log \rho_{N,\varepsilon}) 
      \chi \Big( \frac{G_{N,\varepsilon_0}}{M} \Big)  \nonumber \\
  &&  +  \int_{\Pi^{dN}} \varphi \rho_{N,\varepsilon} (\nabla \log G_{N,\varepsilon}  - \nabla \log \rho_{N,\varepsilon}) 
     \Big(1- \chi \Big( \frac{G_{N,\varepsilon_0}}{M} \Big)\Big) \nonumber  \\
 = &&  I_{1,\varepsilon} + I_{2,\varepsilon}. \nonumber 
\end{eqnarray}
We now have as before the convergence of $I_{1,\eps}$ to
\[
I_1=\int_{\Pi^{dN}} \varphi \rho_{N} (\nabla \log G_{N}  - \nabla \log \rho_{N}) 
      \chi \Big( \frac{G_{N,\varepsilon_0}}{M} \Big).
      \]
      More precisely $I_{1,\varepsilon}$ converges to $I_1$ since as before $G_{\varepsilon_0,M} \le 2 M$  provides a uniform lower bound on $\inf_{i\neq j} |x_i-x_j|$.

Concerning $I_{2,\varepsilon}$, we remark that 
\begin{eqnarray}
 |I_{2,\varepsilon}|
 && \le
 \int_{G_{N,\varepsilon_0} \ge M} \rho_{N,\varepsilon} \bigl|\nabla \log G_{N,\varepsilon}
     - \nabla \log \rho_{N,\varepsilon} \bigr|    \nonumber \\   
&&   \le  \Bigl(\int_{\Pi^{dN}} \rho_{N,\varepsilon} \bigl|\nabla \log \frac{G_{N,\varepsilon}}{\rho_{N,\varepsilon}}|^2
    \Bigr)^{1/2} \Bigl( \int_{\Pi^{dN}} \rho_{N,\varepsilon} \, \mathbb{I}_{G_{N,\varepsilon} \ge M} \Bigr)^{1/2}  \nonumber \\
&& \le C  \frac{1}{(\log M)^{1/2}} \Bigl( \int_{\Pi^{dN}} \rho_{N,\varepsilon} \log \rho_{N,\varepsilon}  \Bigr)^{1/2}. \nonumber\\ 
&& \le \frac{C}{(\log M)^{1/2}},
\end{eqnarray}
with $C$ independent on $\varepsilon$. Now letting $\eps_0\to 0$ and hence $M\to \infty$, we conclude that
\[
\int_{\Pi^{dN}} \varphi \rho_{N,\varepsilon} \big(\nabla \log G_{N,\varepsilon}  - \nabla \log \rho_{N,\varepsilon}\big)\to \int_{\Pi^{dN}} \varphi \rho_{N} \big(\nabla \log G_{N}  - \nabla \log \rho_{N}\big),
\]
and hence that $\rho_N$ solves \eqref{liouvilleN} in the sense of distribution.
\end{proof}

%%%%%%%%%%%% CONCLUSION
\section*{Conclusion}
We have been able to derive for the first time the mean field limit for attractive singular interaction of gradient flow type. Our approach relies critically on the use of the right physics through the structure of the free energy of the system which allows to combine the two methods in \cite{JaWa2}  and \cite{Se}. An important application is the answer to the longstanding open problem of the full rigorous derivation with quantitative estimates of the Patlak--Keller-Segel model in the optimal subcritical regime.

%The crucial observation is that when combining the relative entropy and the modulated energy in an appropriate manner and differentiating in time, bad quantities in previous calculations cancel allowing to obtain the important identity \eqref{ModulatedFreeEnergy}.

Controlling the time evolution of our modulated free energy leads to the development of new large deviation estimates that encode the competition between diffusion and attraction or concentration of the particles. Those large deviation estimates only require simple one-sided bound on the potential near its singularity at $0$ without strong structural assumptions.

By adding to the proofs in the present paper for attractive kernels the estimates in the proceeding \cite{BrJaWa1} for repulsive kernels, our method provides quantitative mean field estimates for a large class of attractive-repulsive interactions. Furthermore, our approach is compatible with vanishing or degenerate diffusion systems where the diffusion coefficient vanishes as the number of particles increases and which are especially relevant for some Coulomb gases related to the complex Ginibre Ensemble in random matrix theory, see for example \cite{BolChaFon}.

\medskip

Finally, we wish to highlight the following open questions for which this new method could be helpful:
\begin{itemize}
\item It is unclear what the optimal rate of convergence should be in Theorem~\ref{Main} and additional work is still needed. In particular the proof in section~\ref{LDE} is somewhat careless in that regard for the sake of simplicity. In particular the use of different regularizations in Prop. \ref{largedeviationbasic}, Prop.~\ref{largedeviation} and the proof of Prop.~\ref{prop3.1} likely leads to an artificially lower rate. However this proof still suggests that a polynomial rate in $N$ cannot be uniformly maintained as the potential $V$ approaches the critical case $\lambda\,\log |x|$ with $\lambda=2\,d\,\sigma$. 
\item Is it possible to obtain uniform in time convergence or in general to work on the infinite time interval $[0,+\infty)$? This would both provide the mean field limit and the  large time asymptotics of the dynamics. In that regards, we point out that our proof does not use the dissipation term in the time evolution of the modulated free energy  in Inequality \eqref{ModulatedFreeEnergy}. Of course this dissipation term is an equivalent of a weighted and modulated Fisher information and it is technically challenging to use because it involves the singular Gibbs equilibrium. But a good entropy-entropy dissipation estimate for this term could lead to uniform in time estimates.
\item Does the mean field limit hold in the supercritical cases? In that case the limiting Patlak-Keller-Segel system blows-up in finite time but it would not be unreasonable to conjecture that the limit holds on the time interval before the blow-up. There are again significant technical issues (including at the level of the existence of our entropy solutions) but we believe that it is possible to develop localized relative entropies that would allow such a result.
\item Can our method provide some insights for the existence of strong stochastic solutions to the trajectorial many-particle system~\eqref{sys}? The entropy solution that we derive in the appendix is so far the only example of some sort of existence outside of the diffusion dominant result $\lambda<\sigma$  as in \cite{FoJo}. Such entropy solutions rely on a (simplified) version of the large deviation inequalities which could provide further insights into whether particles can actually collide and how they collide.
    \item Can the modulated free energy be extended to different type of interactions than gradient flows such as Hamiltonian systems? Another example is given by interactions between particles that solve evolution in time equation, with in particular the Patlak--Keller--Segel parabolic-parabolic equations.
\end{itemize}
%%%%%%%%ACKNOWLEDGEMENTS
\section*{Acknowledgements}
The first two  authors want to thank F. Golse and L. Saint-Raymond for remarks on the note published in the C.R. Acad Sciences Section Math and on the paper https://slsedp.centre-mersenne.org/journals/SLSEDP/  published  by the mersenne foundation. All authors want to thank  S. Serfaty for sharing many insights on her results and for the description of our method in her accepted  paper in {\it Duke J. Math} (2020).  D. Bresch is partially supported by SingFlows project, grant ANR-18-CE40-0027. P.--E. Jabin is partially supported by NSF DMS Grant 161453, 1908739, 2049020. Z. Wang is partially supported by the start-up fund from BICMR, Peking University.

%\nocite{*}

%%%%%%%%%%%%%%%% BIBLIO
 
\end{document}